\documentclass[10pt]{article}

\usepackage[margin=1in]{geometry}
\usepackage{latexsym, amssymb, amsfonts, amsthm, graphics, graphicx, bbm}
\usepackage{amsmath}
\usepackage{epstopdf}
\usepackage[pdftex,colorlinks,linkcolor=blue,menucolor=red,backref=false,bookmarks=true]{hyperref}
\usepackage{cleveref}
\usepackage{mathtools}
\usepackage{subfig}
\usepackage{xcolor}
\usepackage[font=small,labelfont=bf]{caption}
\usepackage{orcidlink}
\usepackage{newpxtext}
\usepackage{enumitem}
\usepackage{array}
\usepackage{multicol}
\usepackage{multirow}
\usepackage{tikz}
\usetikzlibrary{shapes,arrows,positioning, arrows.meta, patterns}
\usepackage[T1]{fontenc}
\usepackage{pifont}
\usepackage{adjustbox}
\usepackage{authblk}

\usepackage{mathtools}
\usepackage{pgfplots}
\pgfplotsset{compat=1.18}

\newtheorem{theorem}{Theorem}[section]

\newtheorem{corollary}[theorem]{Corollary}

\newtheorem{definition}[theorem]{Definition}
\newtheorem{example}[theorem]{Example}

\newtheorem{lemma}[theorem]{Lemma}

\newtheorem{problem}[theorem]{Problem}
\newtheorem{proposition}[theorem]{Proposition}
\newtheorem{remark}[theorem]{Remark}

\newcommand{\norm}[1]{\left\lVert #1 \right\rVert}
\newcommand{\field}[1]{\mathbb{#1}}
\newcommand{\R}{\field{R}}
\newcommand{\N}{\field{N}}

\newcommand{\tr}{\mathrm{tr}}

\newcommand{\Law}{\mathcal{L}}
\newcommand{\ms}{\bar \sigma}

\DeclareMathOperator{\diag}{diag}
\DeclareMathOperator{\row}{row}
\DeclareMathOperator{\col}{col}
\DeclareMathOperator{\con}{con}

\DeclareMathOperator{\spec}{spec}

\numberwithin{equation}{section}

\title{Anchoring and Mixed-Norm Contractions in Averaging-Learning Dynamics}

\author[1,2]{Ionel Popescu}
\author[3]{Jeven Syatriadi}
\author[ ]{Tushar Vaidya}

\affil[1]{University of Bucharest, Faculty of Mathematics and Computer Science, Bucharest, Romania}
\affil[2]{Simion Stoilow Institute of Mathematics of the Romanian Academy, Bucharest, Romania}
\affil[3]{Division of Mathematical Sciences, Nanyang Technological University, Singapore}

\date{}

\begin{document}
\maketitle

\begin{abstract}
A single informed agent can draw an arbitrarily large network to the ground truth. This is the sharpest consequence of the ``Averaging plus Learning'' framework studied here, where agents update opinions by socially averaging neighbours while some receive private feedback at heterogeneous rates. The key is a graph-theoretic property we call \emph{condensely anchored}, which implies convergence to the correct consensus on fixed networks.
In the original framework of Popescu and Vaidya (2023), every agent was required to learn. Removing that requirement changes the problem fundamentally: the underlying graph must now carry the signal from a handful of anchors to everyone else. When learning rates decay to zero, a persistence condition on the rates alone suffices, with no uniform connectivity or aperiodicity assumed. The hardest case is intermittent connectivity, where no single time step contracts in any standard norm. A mixed-operator-norm framework is developed that extracts two-step contraction from the interplay between aggregate learning mass and entrywise diffusion of influence, a mechanism new to consensus literature. Finally, we demonstrate the framework's robustness: vanishing noise preserves convergence to the ground truth, whereas persistent noise drives the system to a limiting law.
\end{abstract}




\tableofcontents


\section{Introduction}
The DeGroot model \cite{degroot1974reaching} has become a workhorse for understanding how consensus forms in networked systems \cite{coculescu2024opinion,elboim2025edge,ghosh2022synchronized,molavi2018theory,sargent2024economic,sugishita2021opinion}. In this framework, each agent repeatedly averages the opinions of their neighbors, leading the group to a common belief. While simple, this model touches on deep mathematical questions about the convergence of matrix products, with foundational work showing that consensus is guaranteed by the weak ergodicity of the underlying influence matrices \cite{chatterjee1977towards, seneta2006non,touri2013product}.

However, mathematical convergence does not guarantee a desirable outcome.  Indeed, foundational models of social learning show that even rational agents may fail to aggregate information efficiently, leading to persistent mislearning or consensus on the wrong belief \cite{acemoglu2011opinion}. Furthermore, a group's final opinion can be disproportionately swayed by a few highly influential agents, regardless of the accuracy of their initial beliefs \cite{golub2010naive}. This creates a tension: mathematics tells us when agreement is possible, but this agreement may be fragile, biased, or systematically wrong.

To address these limitations, the broader literature has explored several paths. One path focuses on outcomes that diverge entirely from a single consensus, with models of signed networks explaining opinion polarization \cite{tian2024spreading} and others showing how opinions can drift erratically without ever converging \cite{rabinovich2021erratic}. A second, highly influential path introduces the concept of ``anchoring," most notably in the Friedkin-Johnsen model \cite{friedkin1990social}, where agents remain partially tethered to their initial, private opinions. This elegantly explains persistent disagreement but frames the anchor as an internal, subjective bias. Our framework addresses the problem of tracking a known ground truth in a deterministic system. This is distinct from the distributed optimization literature, where agents aim to find an unknown optimum. For instance, the work of \cite{reisizadeh2025almost} proves almost sure convergence to a minimizer by using a stochastic gradient method, with convergence guarantees based on conditions on the step-sizes. In contrast, we prove asymptotic convergence by deriving conditions directly on the persistence of the learning signal itself.

Recent work in opinion dynamics has made significant progress in understanding how societies can converge to truth through social learning: \cite{glass2021opinion} analyze how truth-seeking agents interact with social influence and conflicting information sources in bounded confidence models, establishing conditions under which consensus on truth can emerge despite the presence of competing misinformation. Comprehensive surveys \cite{bernardo2024bounded, dong2018survey} document the rich landscape of consensus mechanisms and their convergence properties. However, most existing analyses operate under restrictive assumptions: static network topologies where interaction patterns remain fixed, constant learning rates where agents weigh information uniformly over time, or fixed confidence bounds that don't adapt to uncertainty. In many real world settings, from scientific communities updating beliefs as evidence accumulates, to distributed sensor networks with time-varying communication links, to social networks where attention and trust evolve, these assumptions are violated.
This paper addresses the question: Under what conditions do agents on time-varying networks, using vanishing step sizes typical of stochastic approximation, converge not just to consensus but to correct consensus when truth signals are available? This setting combines three elements that have been largely studied in isolation:

\begin{enumerate}
    \item Time-varying topology: Network structure evolves, requiring joint connectivity conditions \cite{Jadbabiecoord2003}.
    \item Vanishing learning rates: Standard in stochastic approximation \cite{robbins1951stochastic} but rarely analyzed with consensus dynamics. In this document, vanishing learning rates are a modelling feature.
    \item External truth signals: Present in opinion dynamics \cite{fagnani2009average,glass2021opinion} but typically with fixed weights.
\end{enumerate}

The interplay creates technical challenges that neither classical consensus theory \cite{chatterjee1977towards, touri2013product} nor standard stochastic approximation directly address. Our main contribution is establishing convergence guarantees under explicit conditions on \emph{network connectivity, signal quality, learning rate schedules} that bridge these literatures.

Our work builds on this and motivates our central research question: can the anchoring mechanism be repurposed to lead a group not just toward agreement, but toward an objective, external ground truth? We address this question by studying an ``Averaging plus Learning" model that blends social influence with direct feedback from an external signal $\ms$ \cite{popescu2023averaging}:
\begin{equation} \label{eq:average+learn}
X_{t+1} = \underbrace{A_t X_t}_{\text{averaging}} + \underbrace{\mathcal{E}_t(\ms - X_t)}_{\text{learning}}.
\end{equation}

The first term  $A_t X_t$, where $A_t$ is a \emph{row-stochastic} matrix, represents the standard DeGroot-style social learning, where agents average the opinions of their neighbors.
The second term $\mathcal{E}_t (\ms -X_t)$ introduces an external learning mechanism.
Here, the vector $\ms -X_t$ can be interpreted as a private signal each agent receives, indicating their personal error relative to the ground truth $\ms$.
The response to this signal is modulated by the diagonal matrix $\mathcal{E}_t$ which represents the heterogeneous learning ability of the agents, reflecting their varying access to, or trust in, the external information.
This framework, depicted in Figure~\ref{fig:avglearn}, allows us to explore how truth can propagate through a network even when learning ability is limited or intermittent.

Our primary contribution is to establish rigorous and checkable conditions for convergence to the ground truth, relaxing common assumptions on network connectivity and learning behavior, particularly in the challenging context of time-varying networks \cite{Jadbabiecoord2003}.
\begin{itemize}
\item For the \textbf{time-invariant case}, spectral radius arguments show that consensus can be reached even when some agents are defective (i.e., having zero learning rates), as long as the network is \emph{condensely anchored} to the balanced learners.

\item For the more complex \textbf{time-varying case}, where spectral methods fail, we develop a novel framework using \textbf{mixed-operator norms} to establish verifiable entrywise and trace bounds that guarantee contraction.

\item We analyze a \textbf{vanishing-learning regime} where $\mathcal{E}_t(i) \to 0$. We prove that a minimal persistence condition, $\sum_t \min_i \mathcal{E}_t(i)=\infty$, is sufficient for convergence. This is significant as it covers intermittently disconnected topologies and learning rates that may be zero infinitely often.
\end{itemize}

Our focus on these asymptotic convergence properties provides a complementary perspective to studies that analyze the transient, finite-time dynamics of influence systems, such as the preservation of the 'wisdom of the crowd' effect \cite{bullo2020finite}.

The paper is organized as follows. 
Section~\ref{subsec:contributions} lays out a guide of the how the analysis progresses.
Section~\ref{subsec:noteprelim} establishes our notation.
Section~\ref{sec:timeinvar} treats the time-invariant model. \Cref{sec:timvar} explores the time-varying case through the vanishing learning regime and the mixed-operator-norms framework. In \Cref{sec:noise} we treat the case of vanishing noise and also persistent noise.    

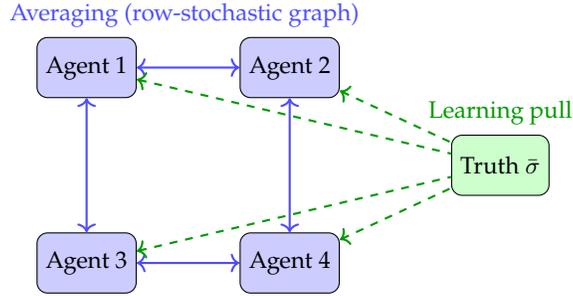
\begin{figure}[htbp]
\centering
\begin{tikzpicture}[scale=1.0, every node/.style={font=\small}]

    \node[rectangle, rounded corners, draw, fill=blue!20, minimum size=0.8cm] (x1) at (-0.5, 1.3) {Agent 1};
    \node[rectangle, rounded corners, draw, fill=blue!20, minimum size=0.8cm] (x2) at (2.2, 1.3) {Agent 2};
    \node[rectangle, rounded corners, draw, fill=blue!20, minimum size=0.8cm] (x3) at (-0.5, -1.3) {Agent 3};
    \node[rectangle, rounded corners, draw, fill=blue!20, minimum size=0.8cm] (x4) at (2.2, -1.3) {Agent 4};

    \node[rectangle, rounded corners, draw, fill=green!20, minimum width=1.2cm, minimum height=0.8cm] (truth) at (5, 0) {Truth $\ms$};

    \draw[<->, thick, blue!70] (x1) -- (x2);
    \draw[<->, thick, blue!70] (x1) -- (x3);
    \draw[<->, thick, blue!70] (x2) -- (x4);
    \draw[<->, thick, blue!70] (x3) -- (x4);

    \draw[<-, dashed, thick, green!60!black] (x1) -- (truth);
    \draw[<-, dashed, thick, green!60!black] (x2) -- (truth);
    \draw[<-, dashed, thick, green!60!black] (x3) -- (truth);
    \draw[<-, dashed, thick, green!60!black] (x4) -- (truth);
    
    \node at (0.8, 2) {\textcolor{blue!70}{Averaging (row-stochastic graph)}};
    \node at (5, 0.7) {\textcolor{green!60!black}{Learning pull}};
\end{tikzpicture}
\caption{Averaging plus learning model. Solid blue edges represent averaging across a fixed influence network. Dashed green arrows represent learning pulls from the ground truth $\ms$.}
\label{fig:avglearn}
\end{figure}

\subsection{Summary of contributions}
\label{subsec:contributions}
Our results are organised along a chain of increasingly general settings, summarised
in the diagram below. In the time-invariant case, we identify a graph-theoretic property,
\emph{condensely anchored}, and show it is equivalent to zero-convergence when
$A - \mathcal{E} \geq 0$. 
More generally, under anchoring interval assumptions, condensely anchored implies eventual consensus for agents.
For time-varying networks, we develop two complementary approaches: a
persistence argument that requires only $\sum_t \min_i \mathcal{E}_t(i) = \infty$ with no topological
assumptions, and a mixed-operator-norm framework that extracts strict two-step contraction
even when no single step contracts. The stochastic extensions show that vanishing noise
preserves convergence while persistent noise drives the process to a limiting law.
\Cref{fig:summary} below traces the logical
progression of the main results.
\medskip

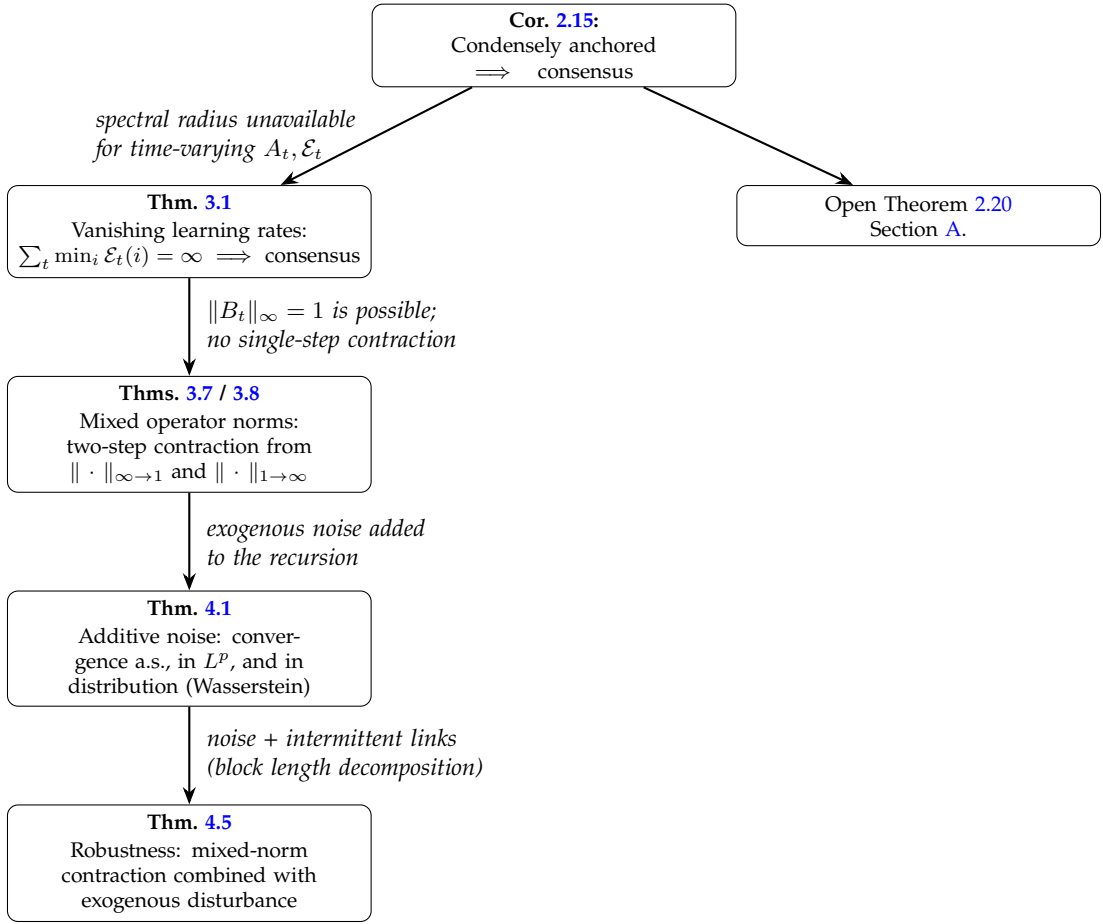
\begin{figure}
\begin{center}
\begin{tikzpicture}[
    node distance=1.3cm and 0cm,
    block/.style={
        rectangle, draw, rounded corners,
        text width=4.6cm, align=center,
        minimum height=0.8cm, font=\footnotesize
    },
    reason/.style={
        font=\small\itshape, text width=5cm, align=left
    },
    arr/.style={-{Stealth}, thick}
]

\node[block] (T1) {%
    \textbf{Cor.~\ref{cor:condanchor}:}\\ Condensely anchored
    $\implies$ consensus
};

\node[block, below right=of T1] (T6) {%
    Open \Cref{prob:genconstantbound} \\ \Cref{appendix}.
};

\node[block, below left=of T1] (T2) {%
    \textbf{Thm.~\ref{thm:main_vanish}}\\[1pt]
    Vanishing learning rates:
    \(
    \sum_t \min_i \mathcal{E}_t(i) = \infty \implies \text{consensus}
    \)
};

\node[block, below=of T2] (T3) {%
    \textbf{Thms.~\ref{thm:normcond} / \ref{thm:normcondgeneral}}\\[1pt]
    Mixed operator norms:\\
    two-step contraction from $\|\cdot\|_{\infty \to 1}$ and $\|\cdot\|_{1 \to \infty}$
};

\node[block, below=of T3] (T4) {%
    \textbf{Thm.~\ref{t:35}}\\[1pt]
    Additive noise: convergence a.s., in $L^p$,
    and in distribution (Wasserstein)
};

\node[block, below=of T4] (T5) {%
    \textbf{Thm.~\ref{thm:norm:noisy-prob}}\\[1pt]
    Robustness: mixed-norm contraction 
    combined with exogenous disturbance
};

\draw[arr] (T1) -- (T2)
    node[reason, midway, left, xshift=40pt] {spectral radius unavailable\\for time-varying $A_t, \mathcal{E}_t$};

\draw[arr] (T2) -- (T3)
    node[reason, midway, right, xshift=3pt] {$\|B_t\|_\infty = 1$ is possible;\\no single-step contraction};

\draw[arr] (T3) -- (T4)
    node[reason, midway, right, xshift=3pt] {exogenous noise added\\to the recursion};

\draw[arr] (T4) -- (T5)
    node[reason, midway, right, xshift=3pt] {noise + intermittent links\\(block length decomposition)};

\draw[arr] (T1) -- (T6);
\end{tikzpicture}
\end{center}
\caption{Summary of contributions}
\label{fig:summary}
\end{figure}

\subsection{Notation and preliminaries}
\label{subsec:noteprelim}

Let $n$ be a positive integer.
We denote the identity matrix of size $n$ and the zero square matrix of size $n$ by $I_n$ and $0_n$, respectively.
The all-ones vector of length $n$ is denoted by $\mathbf{1}_n$.
If it is clear from the context, we may omit the subscript.
We also use $0$ to denote the zero vector.
Let $M$ be a real square matrix of size $n$.
For each $i,j \in \{1,2,\dots,n\}$, the entry of $M$ on the $i$-th row and the $j$-th column is denoted by $M(i,j)$.
Let $i \in \{1,2,\dots,n\}$.
The $i$-th row of $M$ is denoted by $\row_i(M)$, which is a row vector of length $n$.
The $i$-th row-sum of $M$ refers to the sum of the entries of $\row_i(M)$.
Meanwhile, the $i$-th column of $M$ is denoted by $\col_i(M)$, which is a column vector of length $n$.
If $M$ is a diagonal matrix of size $n$, we use $M(i)$ to denote the $i$-th diagonal entry of $M$.
For brevity, we alternatively write $M = \diag(M(1), M(2), \dots, M(n))$.
Let $\mathbf{v}$ be a vector in $\mathbb{R}^n$ and for each $i \in \{1,2,\dots,n\}$, let $\mathbf{v}(i)$ denote the $i$-th entry of $\mathbf{v}$.

Let $M_1$ and $M_2$ be real square matrices of size $n > 0$.
We write $M_1 \leq M_2$ if $M_1$ is less than or equal to $M_2$ entrywise.
Similarly, we write $M_1 \geq M_2$ if $M_1$ is greater than or equal to $M_2$ entrywise.
If a real square matrix $M$ satisfies $M \geq 0$, then we call $M$ a \textbf{nonnegative matrix}.
If $\mathbf{v}$ is a vector in $\mathbb{R}^n$, then we write $\mathbf{v} \geq 0$ to indicate that all entries of $\mathbf{v}$ are nonnegative.

If $x \in \mathbb{C}$ then $\lvert x \rvert$ denotes the modulus of $x$.
Let $M$ be a complex square matrix of size $n > 0$.
Let $\lvert M \rvert$ be the real square matrix of size $n$ such that $\lvert M \rvert(i,j) = \lvert M(i,j) \rvert$ for all $i,j \in \{1,2,\dots,n\}$.
If $\lambda_1$, $\lambda_2$, $\dots$, $\lambda_n$ are all of the eigenvalues of $M$, then the \textbf{spectral radius} of $M$ is defined to be
\[
\rho(M) \coloneqq \max \{ \lvert \lambda_i \rvert : 1 \leq i \leq n \}.
\]
By \cite{DP65} or \cite[Theorem 2.21]{Varga00}, we have that $\rho(M) \leq \rho(\lvert M \rvert)$.
We call $M$ \textbf{zero-convergent} if $\displaystyle \lim_{k \to \infty} M^k = 0$.
We have the following well-known characterization of zero-convergent matrices.
\begin{theorem}[{\cite[Theorem 5.6.12]{HJ85}}]
\label{thm:zeroconv}
    Let $M$ be a complex square matrix.
    Then $M$ is zero-convergent if and only if $\rho(M)<1$.
\end{theorem}

Let $M$ be a real square matrix of size $n>0$ and let $\mathbf{v}$ be a vector in $\mathbb{R}^n$.
If $p \geq 1$ is a real number, then the $p$-norm of $\mathbf{v}$ is
\[
\norm{\mathbf{v}}_p = \left( \sum_{i=1}^n \lvert \mathbf{v}(i) \rvert^p \right)^{1/p},
\]
while the $\infty$-norm of $\mathbf{v}$ is $\displaystyle \norm{\mathbf{v}}_{\infty} = \max_{1 \leq i \leq n} \lvert \mathbf{v}(i) \rvert$.
We define the induced matrix $p$-norm as
\[
\lVert M \rVert_p \coloneqq \sup_{\mathbf{v} \in \mathbb{R}^n \backslash \{0\}} \frac{\lVert M \mathbf{v} \rVert_p}{\lVert \mathbf{v} \rVert_p},
\]
where $1 \leq p \leq \infty$.
In particular, if $p=1$ or $p=\infty$, we have the following formulae of the induced matrix $1$-norm and $\infty$-norm:
\begin{align*}
    \norm{M}_1 &= \max_{1 \leq i \leq n} \left( \textbf{1}_n^\top \, \col_i( \lvert M \rvert ) \right),\\
    \norm{M}_{\infty} &= \max_{1 \leq i \leq n} \left( \row_i( \lvert M \rvert ) \, \textbf{1}_n \right).
\end{align*}
We also have the inequality $\rho(M) \leq \norm{M}_p$ for any $1 \leq p \leq \infty$.

Let $M$ be a nonnegative matrix of positive size.
We use the following definitions:
\begin{itemize}
    \item $M$ is \textbf{row stochastic} if each row-sum of $M$ is equal to $1$.
    \item $M$ is \textbf{row substochastic} if each row-sum of $M$ is at most $1$.
    \item $M$ is \textbf{proper substochastic} if $M$ is row substochastic and at least one row-sum of $M$ is less than $1$.
\end{itemize}
Hence, it is clear that the set of row-substochastic matrices consists of row-stochastic and proper-substochastic matrices.
If $M$ is row substochastic then $\rho(M) \leq 1$.
In particular, if $M$ is row stochastic then $\rho(M) = 1$.

\subsection{Directed graphs and row-stochastic matrices}
\label{subsec:digraphstoch}

In this paper, we require some basic facts about directed graphs.
A \textbf{directed graph} (or \emph{digraph}) $G$ is an ordered pair $(V,E)$ consisting of the vertex set $V$ together with the set of \emph{arcs} $E$.
Each arc is an ordered pair $(i,j)$ where $i, j \in V$.
In our visual representation of digraphs, an arc $(i,j)$ is represented by an arrow starting from the \emph{initial vertex} $i$ and ending at the \emph{terminal vertex} $j$.
The arc $(i,i)$ is called a \emph{self-loop} at vertex $i$.
The \emph{outdegree} of a vertex $v \in V$ is the number of arcs $(i,j) \in E$ such that $v=i$.
Likewise, the \emph{indegree} of a vertex $v \in V$ is the number of arcs $(i,j) \in E$ such that $v=j$.
Let $G=(V, E)$ be a digraph and let $W \subseteq V$.
The digraph \emph{induced} by $W$ is the digraph $G \vert_W = (W, E \vert_W)$ where $E \vert_W$ is the set of arcs $(i, j)$ such that $(i, j) \in E \vert_W$ if and only if $i, j \in W$.

Let $k \geq 0$ be an integer.
A \textbf{walk} $\omega$ starting from vertex $i_0$ and ending at vertex $i_k$ in a digraph $G=(V,E)$ is a sequence of vertices $\langle i_0, i_1, \dots, i_k \rangle$ such that $(i_t, i_{t+1}) \in E$ for all $t \in \{0,1,\dots,k-1\}$.
The \emph{length} of the walk $\omega = \langle i_0, i_1, \dots, i_k \rangle$, denoted by $\lvert \omega \rvert$, is equal to $k$.
We can also write the walk $\omega = \langle i_0, i_1, \dots, i_k \rangle$ as $i_0 \to i_1 \to \dots \to i_k$.
Note that a walk may include a self-loop, or a vertex can be visited more than once.
A walk that starts and ends at the same vertex is called a \textbf{closed walk}.
For example, we consider any walk of length 0 to be a closed walk.
A \textbf{cycle} is a closed walk of positive length with no repeated vertex.
We consider a self-loop as a cycle of length $1$.
A directed graph with no cycle is called a \textbf{directed acyclic graph (DAG)}.

Let $G=(V,E)$ be a digraph.
We define a relation on $V$ where $i \sim j$ if and only if there is a walk in $G$ from $i$ to $j$, and there is also a walk from $j$ to $i$.
This relation is an equivalence relation on $V$.
An equivalence class induced by this equivalence relation is called a \textbf{strongly connected component (SCC)} of $G$.
If $G$ consists of only one SCC, then we call $G$ \textbf{strongly connected}.
Equivalently, $G$ is strongly connected if and only if for all $i,j \in V$, there exists a walk in $G$ from $i$ to $j$.
Let $W$ be an SCC of $G$.
Clearly, the induced digraph $G \vert_W$ is strongly connected.
We call $W$ a \textbf{sink SCC} if for all $(i,j) \in E$, $i \in W \implies j \in W$.
In words, $W$ is a sink if there is no arc with its initial vertex in $W$ and its terminal vertex outside of $W$.
For example, any vertex that has outdegree zero forms a sink SCC.

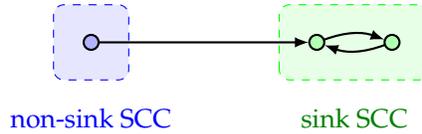
\begin{figure}[htbp]
    \centering
    \begin{tikzpicture}
    \draw[rounded corners, fill=blue!10, draw=blue, dashed] (-0.5, -0.5) rectangle (0.5, 0.5);
    \node[circle, fill=blue!30, draw=black, thick, scale=0.6] (a) at (0, 0) {};
    
    \draw[rounded corners, fill=green!10, draw=green, dashed] (2.5, -0.5) rectangle (4.5, 0.5);
    \node[circle, fill=green!30, draw=black, thick, scale=0.6] (b) at (3, 0) {};
    \node[circle, fill=green!30, draw=black, thick, scale=0.6] (c) at (4, 0) {};
    
    \draw[-latex, thick] (a) -- (b);
    \draw[-latex, thick] (b) to [bend left=20] (c);
    \draw[-latex, thick] (c) to [bend left=20] (b);
    
    \node[blue] at (0, -1) {non-sink SCC};
    \node[green!50!black] at (3.5, -1) {sink SCC};
    \end{tikzpicture}
    \caption{A sink SCC has no arc leaving it.}
\end{figure}

The SCCs partition the vertex set of $G$ and they also form the \emph{condensation digraph} of $G$.
The \textbf{condensation} of $G$ is the directed graph $G^\circ = (V^\circ, E^\circ)$ where 
\begin{itemize}
    \item $V^\circ$ is the set of all SCCs of $G$.
    \item For all distinct $W_1, W_2 \in V^\circ$, we have $(W_1, W_2) \in E^\circ$ if there exists $w_1 \in W_1$ and $w_2 \in W_2$ such that $(w_1, w_2) \in E$.
\end{itemize}
Observe that if $(W_1, W_2) \in E^\circ$ then $(W_2, W_1) \notin E^\circ$.
As an example, the digraph in \cite[Figure 2]{DPRRW24} is the condensation of the digraph in \cite[Figure 1]{DPRRW24}.
Additionally, the green vertex in \cite[Figure 1]{DPRRW24} forms the only sink SCC of the digraph.
In Figure~\ref{fig:condensation_ex}, we give an example of a digraph $G$ and its condensation $G^\circ$.
The vertices of $G^\circ$ are the three SCCs of $G$, and $\{3,4\}$ is the only sink SCC of $G$.
The condensation digraph is a DAG.
Algorithmically, finding SCCs can be efficiently done with linear time complexity.
Some well-known algorithms are due to Tarjan~\cite{Tarjan72}, Dijkstra~\cite[Chapter 25]{Dijkstra76} (with many variants), and Kosaraju--Sharir~\cite{CLRS09}.

\begin{figure}[htbp]
    \centering
    \begin{minipage}{.3\textwidth}
        \begin{tikzpicture}
        \node[circle, fill=violet, draw=black, thick, scale=0.8, label=below:1] (v1) at (0, 0) {};
        \node[circle, fill=lightgray, draw=black, thick, scale=0.8, label=below:2] (v2) at (-1.5, -1.5) {};
        \node[circle, fill=green, draw=black, thick, scale=0.8, label=left:3] (v3) at (-3, 0) {};
        \node[circle, fill=green, draw=black, thick, scale=0.8, label=right:4] (v4) at (-1.5, 1.5) {};
        \node at (0,-1.7) {\large $G$};

        \draw[-latex, thick] (v2) -- (v1);
        \draw[-latex, thick] (v1) -- (v3);
        \draw[-latex, thick] (v4) to [bend left=15] (v3);
        \draw[-latex, thick] (v3) to [bend left=25] (v4);
        \draw[-latex, thick] (v2) -- (v3);
        \draw[-latex, thick] (v1) -- (v4);
        \draw[-latex, thick] (v4) to [out=130,in=50, loop, style={min distance=8mm}] (v4);
        \draw[-latex, thick] (v1) to [out=-40,in=40, loop, style={min distance=8mm}] (v1);
        \end{tikzpicture}
    \end{minipage}
    \hspace{2.8 cm}
    \begin{minipage}{.25\textwidth}
        \vspace{0.6 cm}
        \begin{tikzpicture}
        \node[circle, fill=violet, draw=black, thick, scale=0.8, label=above:\text{\{1\}}] (v1) at (0, 0) {};
        \node[circle, fill=lightgray, draw=black, thick, scale=0.8, label=below:\text{\{2\}}] (v2) at (-1.5, -1.5) {};
        \node[circle, fill=green, draw=black, thick, scale=0.8, label=above:\text{\{3,4\}}] (v3) at (-3, 0) {};
        \node at (0,-2.5) {\large $G^\circ$};

        \draw[-latex, thick] (v2) -- (v1);
        \draw[-latex, thick] (v1) -- (v3);
        \draw[-latex, thick] (v2) -- (v3);
        \end{tikzpicture}
    \end{minipage}
    \caption{Example of a digraph $G$ and its condensation $G^\circ$.}
    \label{fig:condensation_ex}
\end{figure}
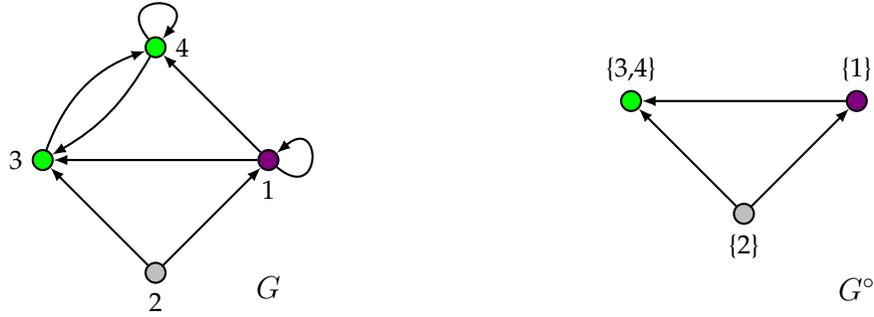

Let $G$ be a strongly-connected digraph that contains at least one cycle.
The \textbf{period} $p$ of $G$ is the greatest common divisor of the lengths of all cycles in $G$.
If $p=1$ then we call $G$ \textbf{aperiodic}.
For example, if $G$ contains a self-loop then $G$ is aperiodic.
If $G$ is strongly connected that is also a DAG, then $G$ is just a single vertex without a self-loop.
See \cite[Section~13.3]{JS99} or \cite{Pascoletti92} for some efficient algorithms to compute the period of a strongly-connected digraph without having to find all of its cycles.
We caution the readers that for general digraphs, including those that are not strongly connected, there are some variations in defining periodicity or aperiodicity.
For example, see \cite{KR26} for ``Boolean" periodicity, and see \cite[Definition~2]{golub2010naive} for ``global" aperiodicity.
There should be no confusion when dealing with strongly-connected digraphs.
Next, we call a digraph $G$ \textbf{condensely aperiodic} if for all sink SCC $W$ of $G$, the induced digraph $G \vert_W$ is aperiodic or $G \vert_W$ is a single vertex without a self-loop.
For example, see \Cref{fig:condaper}.

\begin{figure}[htbp]
    \centering
    \begin{minipage}{.31\textwidth}
        \centering
        \begin{tikzpicture}[scale=0.9]
        \node[circle, fill=white, draw=black, thick, scale=0.7, label=above:$1$] (a) at (0, 1.2) {};
        \node[circle, fill=white, draw=black, thick, scale=0.7, label=right:$2$] (b) at (1.2, 0) {};
        \node[circle, fill=white, draw=black, thick, scale=0.7, label=below:$3$] (c) at (0, -1.2) {};
        \node[circle, fill=white, draw=black, thick, scale=0.7, label=left:$4$] (d) at (-1.2, 0) {};
        
        \draw[-latex, thick] (a) -- (b);
        \draw[-latex, thick] (b) -- (c);
        \draw[-latex, thick] (c) -- (d);
        \draw[-latex, thick] (d) -- (a);
        \end{tikzpicture}
        \caption*{The only cycle here is of length 4.}
    \end{minipage}
    \hspace{1cm}
    \begin{minipage}{.25\textwidth}
        \centering
        \begin{tikzpicture}[scale=0.9]
        \node[circle, fill=white, draw=black, thick, scale=0.7, label=above:$1$] (a) at (0, 1.2) {};
        \node[circle, fill=white, draw=black, thick, scale=0.7, label=right:$2$] (b) at (1.2, 0) {};
        \node[circle, fill=white, draw=black, thick, scale=0.7, label=left:$3$] (c) at (-1.2, 0) {};
        
        \draw[-latex, thick, black] (a) to [bend left=20] (b);
        \draw[-latex, thick, black] (b) to [bend left=20] (a);
        
        \draw[-latex, thick, black] (a) -- (c);
        \draw[-latex, thick, black] (c) -- (b);
        \end{tikzpicture}
        \caption*{There is a cycle of length 2 and another one of length 3.}
    \end{minipage}
    \caption{The strongly-connected digraph on the left has period 4, while the strongly-connected digraph on the right is aperiodic.}
\end{figure}

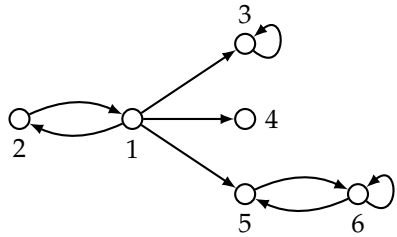
\begin{figure}[htbp]
    \centering
    \begin{tikzpicture}[baseline=0]
        \node[circle, fill=white, draw=black, thick, scale=0.8, label=below:1] (v1) at (0, 0) {};
        \node[circle, fill=white, draw=black, thick, scale=0.8, label=below:2] (v2) at (-3/2, 0) {};
        \node[circle, fill=white, draw=black, thick, scale=0.8, label=above:3] (v3) at (3/2, 1) {};
        \node[circle, fill=white, draw=black, thick, scale=0.8, label=right:4] (v4) at (3/2, 0) {};
        \node[circle, fill=white, draw=black, thick, scale=0.8, label=below:5] (v5) at (3/2, -1) {};
        \node[circle, fill=white, draw=black, thick, scale=0.8, label=below:6] (v6) at (3, -1) {};

        \draw[-latex, thick] (v1) -- (v3);
        \draw[-latex, thick] (v1) -- (v4);
        \draw[-latex, thick] (v1) -- (v5);
        \draw[-latex, thick] (v1) to [bend left=25] (v2);
        \draw[-latex, thick] (v2) to [bend left=25] (v1);
        \draw[-latex, thick] (v5) to [bend left=25] (v6);
        \draw[-latex, thick] (v6) to [bend left=25] (v5);
        \draw[-latex, thick] (v3) to [out=-40,in=40, loop, style={min distance=6mm}] (v3);
        \draw[-latex, thick] (v6) to [out=-40,in=40, loop, style={min distance=6mm}] (v6);
    \end{tikzpicture}
    \caption{An example of a condensely-aperiodic digraph.}
    \label{fig:condaper}
\end{figure}

Let $M$ be a real square matrix of size $n>0$.
The \textbf{underlying digraph} of $M$ is the digraph $G[M]=(V,E)$ where $V = \{1,2,\dots,n\}$ and $(i,j) \in E$ if and only if $M(i,j) \neq 0$.
We call $M$ \textbf{irreducible} if $G[M]$ is strongly connected.
Otherwise, $M$ is \textbf{reducible}.
An irreducible matrix $M$ is \textbf{aperiodic} if $G[M]$ is aperiodic.
We call $M$ \textbf{condensely aperiodic} if $G[M]$ is condensely aperiodic.

Let $A$ be a row-stochastic matrix.
The behavior of $\displaystyle \lim_{t \to \infty} A^t$ can be broadly categorized into two types: 
\begin{enumerate}
    \item $\displaystyle \lim_{t \to \infty} A^t = M$ where $M$ is a row-stochastic matrix.
    \item $A^t$ asymptotically cycles through a finite number of matrices.
\end{enumerate}
The following theorem is a combinatorial characterization of the existence of $\displaystyle \lim_{t \to \infty} A^t$ based on the underlying digraph of a given row-stochastic $A$.
Note that a row-stochastic matrix $A$ is condensely aperiodic if and only if for all sink SCC $W$ of $G[A]$, the induced digraph $G[A] \, \vert_W$ is aperiodic.

\begin{theorem}[{\cite[Theorem 4]{Rosenblatt57}} or {\cite[Theorem 2]{golub2010naive}}]
\label{thm:aperiodic}
    Let $A$ be a row-stochastic matrix of positive size.
    Then $\displaystyle \lim_{t \to \infty} A^t = M$ where $M$ is row stochastic, if and only if $A$ is condensely aperiodic.
\end{theorem}

For row-substochastic matrices, let us consider the next example.
Let
\[
B_1 = \begin{pmatrix}
    0 & 1 \\
    1/2 & 0
\end{pmatrix} \; \text{and} \; B_2 = \begin{pmatrix}
    0 & 1 \\
    1 & 0
\end{pmatrix}.
\]
It is easy to check that the underlying digraphs $G[B_1]$ and $G[B_2]$ are identical.
However, we have that $B_1$ is zero-convergent, while $B_2^t$ periodically equals to $B_2$ or $I_2$.
Despite this, we can state the next theorem.

\begin{theorem} \label{thm:limsubstochastic}
    Let $A$ be a row-substochastic matrix of positive size.
    If $A$ is condensely aperiodic then $\displaystyle \lim_{t \to \infty} A^t = M$ where $M$ is row substochastic.
\end{theorem}

\begin{proof}
    There exists a permutation matrix $P$ such that $PAP^\top$ is equal to the following block matrix:
    \begin{equation} \label{eq:blockform}
        PAP^\top = \begin{pmatrix}
        U & T\\
        0 & S
    \end{pmatrix},
    \end{equation}
    where
    \begin{itemize}
        \item $U$ is an upper-triangular block matrix:
        \[
        U = \begin{pmatrix}
            U_1 & \star & \cdots & \star \\
            0 & U_2 & \cdots & \star \\
            \vdots & \vdots & \ddots & \vdots \\
            0 & 0 & \cdots & U_k
        \end{pmatrix}
        \]
        such that for each $i$, $U_i$ corresponds to a non-sink SCC of $G[A]$.
        Each $U_i$ is irreducible and proper substochastic.
        By \cite[Lemma 2.8]{Varga00}, each $U_i$ is zero-convergent;
        \item $T$ has at least one positive entry (if $T$ is non-empty);
        \item $S$ is a non-empty diagonal block matrix:
        \[
        S = \begin{pmatrix}
            S_1 & 0 & \cdots & 0 \\
            0 & S_2 & \cdots & 0 \\
            \vdots & \vdots & \ddots & \vdots \\
            0 & 0 & \cdots & S_m
        \end{pmatrix}
        \]
        such that for each $i$, $S_i$ corresponds to a sink SCC of $G[A]$.
        Each $S_i$ is irreducible and row substochastic.
        If $S_i$ is proper substochastic then by \cite[Lemma 2.8]{Varga00}, $S_i$ is zero-convergent.
        Otherwise, $\rho(S_i)=1$.
    \end{itemize}
    Suppose that $A$ is condensely aperiodic.
    Then for each $i$, $S_i=0$ or $S_i$ is aperiodic.
    Suppose $S_i$ is aperiodic and thus, primitive.
    Suppose also that $S_i$ is row stochastic, i.e., $\rho(S_i)=1$.
    By the Perron--Frobenius theorem \cite[Theorem 1.1]{seneta2006non}, $S_i$ has $1$ as a simple eigenvalue, and for all eigenvalues $\lambda \neq 1$ of $S_i$, we have $\lvert \lambda \rvert <1$.
    The eigenvalues of $A$ are precisely the eigenvalues of $U_1, U_2, \dots, U_k, S_1, S_2, \dots, S_m$.
    Let $\lambda$ be an eigenvalue of $A$.
    Obviously, $\lvert \lambda \rvert \le 1$.
    Suppose that $\lvert \lambda \rvert = 1$.
    Then at least one of $S_1, S_2, \dots, S_m$ has $\lambda$ as an eigenvalue.
    Suppose that for some $i$, $S_i$ has $\lambda$ as an eigenvalue.
    Since $\lvert \lambda \rvert =1$ then $\rho(S_i)=1$.
    This implies that $S_i$ must be aperiodic and row stochastic.
    Hence, it is clear that $\lambda=1$.
    Furthermore, as an eigenvalue of $A$, $\lambda=1$ is \emph{semisimple}: its algebraic and geometric multiplicities are equal.
    By \cite[Theorem 2.1]{HP12}, we conclude that there exists a matrix $M$ such that $\displaystyle \lim_{t \to \infty} A^t = M$.
    Moreover, observe that all entries of the vector $A^t \mathbf{1}$ is at most $1$, which is denoted as $A^t \mathbf{1} \le \mathbf{1}$.
    Taking limit yields $M \mathbf{1} \le \mathbf{1}$ and thus, $M$ is also row substochastic.
\end{proof}

\section{Time-invariant case and the spectral radius}
\label{sec:timeinvar}

In this section, we consider the time-invariant case of the dynamics \eqref{eq:average+learn}.
Since $A_t \ms = \ms$, we first rewrite Equation~\eqref{eq:average+learn} as follows:
\begin{equation} \label{eq:A+L_rewrite}
    X_{t+1}-\ms = A_t X_t - A_t \ms -\mathcal{E}_t (X_t-\ms) = (A_t-\mathcal{E}_t)(X_t-\ms)
\end{equation}
In the time-invariant case, we have that $A_t = A$ and $\mathcal{E}_t = \mathcal{E}$ for all integers $t \geq 0$.
We fix the notation $A$ and $\mathcal{E}$ for the rest of Section~\ref{sec:timeinvar}, unless stated otherwise.
Next, Equation~\eqref{eq:A+L_rewrite} becomes 
\[
X_{t+1}-\ms = (A-\mathcal{E})(X_t-\ms).
\]
Furthermore, for all integers $t \geq 0$, we have
\begin{equation} \label{eq:A+L_rewrite_constant}
    X_t-\bar{\sigma} = (A-\mathcal{E})^t (X_0-\bar{\sigma}).
\end{equation}
All agents eventually reach consensus means that:
\[
\lim_{t \to \infty} X_t = \ms.
\]
In the next proposition, if $A-\mathcal{E}$ is zero-convergent then $\displaystyle \lim_{t \to \infty} X_t = \ms$, regardless of the initial condition $X_0-\ms$.

\begin{proposition} \label{prop:conv_implies_consensus}
    Suppose that Equation~\eqref{eq:A+L_rewrite_constant} holds for all integers $t \geq 0$.
    If $A-\mathcal{E}$ is zero-convergent then $\displaystyle \lim_{t \to \infty} X_t = \ms$.
\end{proposition}

\begin{proof}
    Suppose that $A-\mathcal{E}$ is zero-convergent, i.e.,
    \[
    \lim_{t \to \infty} (A-\mathcal{E})^t = 0.
    \]
    Since $\displaystyle \lim_{t \to \infty} (A-\mathcal{E})^t$ exists, we conclude that
    \[
    \lim_{t \to \infty} (X_t-\ms) = (X_0-\ms) \lim_{t \to \infty} (A-\mathcal{E})^t = 0 \implies \lim_{t \to \infty} X_t = \ms. \tag*{\qedhere}
    \]
\end{proof}

By Theorem~\ref{thm:zeroconv}, $A-\mathcal{E}$ is zero-convergent if and only if $\rho(A-\mathcal{E}) < 1$.
Even if $\rho(A-\mathcal{E}) \geq 1$, it is possible that consensus can still be eventually reached, depending on $X_0-\ms$.
An example is if $\displaystyle \lim_{t \to \infty} (A-\mathcal{E})^t = M$, where $M \neq 0$ is singular.
Our main focus in this section is to derive various conditions for the entries of $A$ and $\mathcal{E}$ such that $\rho(A-\mathcal{E}) < 1$.

Next, we recall the following result from \cite{popescu2023averaging} that serves as our starting point in this section.
For the readers' convenience, we provide a simplified proof based on the spectral radius and the matrix $\infty$-norm.

\begin{proposition}[{\cite[Proposition~3.1]{popescu2023averaging}}]
\label{prop:positivelearn}
	If $0<\mathcal{E}(i)<2 A(i,i)$ for all $i \in \{1,2,\dots,n\}$, then all agents reach the same consensus value, i.e.,
	\[\lim_{t \to \infty} X_t = \ms.\]
\end{proposition}
\begin{proof}
We know that $\rho(A-\mathcal{E}) \leq \lVert A-\mathcal{E} \rVert_{\infty}$.
Suppose that $0<\mathcal{E}(i)<2 A(i,i)$ for all $i$ in $\{1,2,\dots,n\}$.
We claim that $\lVert A-\mathcal{E} \rVert_{\infty} < 1$.
Since $A$ is row stochastic, for each row $i \in \{1,2,\dots,n\}$, we have that
\[
\sum_{j=1}^{n} \lvert (A-\mathcal{E})(i,j) \rvert = \lvert A(i,i)-\mathcal{E}(i) \rvert + 1-A(i,i).
\]
We have the following equivalence
\[
\lvert A(i,i)-\mathcal{E}(i) \rvert + 1-A(i,i) < 1 \iff \lvert A(i,i)-\mathcal{E}(i) \rvert < A(i,i) \iff 0 < \mathcal{E}(i) < 2 A(i,i).
\]
It follows that $\displaystyle \sum_{j=1}^{n} \lvert (A-\mathcal{E})(i,j) \rvert < 1$ for all $i \in \{1,2,\dots,n\}$, which means that $\lVert A-\mathcal{E} \rVert_{\infty} < 1$.
Thus, we have $\rho(A-\mathcal{E})<1$, and the conclusion follows from Proposition~\ref{prop:conv_implies_consensus}.
\end{proof}

\subsection{Anchor agents and the condensely-anchored property}

First, we introduce the definition of \emph{anchor} and \emph{defective} agents for general time-invariant case.
\begin{definition} \label{dfn:anchor}
  Let $A$ be a row-stochastic matrix and let $\mathcal{E}$ be a diagonal nonnegative matrix of the same size $n > 0$.
  Let $i \in \{1,2,\dots,n\}$.
  An agent $i$ is called an \textbf{anchor} of $A-\mathcal{E}$ if $0 < \mathcal{E}(i) < 2A(i,i)$.
  Alternatively, the row/index/vertex $i$ is also called an anchor of $A-\mathcal{E}$ if $0 < \mathcal{E}(i) < 2A(i,i)$.
  If $\mathcal{E}(i) = 0$ then an agent $i$ is called \textbf{defective}.
  Consequently, an agent $i$ is a non-anchor if either $i$ is defective or $\mathcal{E}(i) \geq 2A(i,i)$.
\end{definition}

We investigate the dynamics of the interactions between anchor and defective agents where consensus is eventually reached.
Proposition~\ref{prop:positivelearn} asserts that if all agents are anchors, then consensus is eventually reached.
In contrast, we will later on derive conditions for zero-convergence where the number of anchors is as low as one.

The next proposition describes the anchors of $A-\mathcal{E}$, where $A-\mathcal{E}$ is nonnegative.

\begin{proposition} \label{prop:anchorAminE>=0}
    Let $A$ be a row-stochastic matrix of size $n > 0$, and let $\mathcal{E}$ be a diagonal nonnegative matrix of size $n$ such that $A-\mathcal{E} \ge 0$.
    Let $i \in \{1,2,\dots,n\}$.
    The following are equivalent:
    \begin{enumerate}[label=(\arabic*)]
        \item The index $i$ is an anchor of $A-\mathcal{E}$.
        \item $\mathcal{E}(i) > 0$.
        \item The $i$-th row-sum of $A-\mathcal{E}$ is strictly less than $1$.
    \end{enumerate}
\end{proposition}

\begin{proof}
    First, note that $0 \leq \mathcal{E}(i) \leq A(i, i)$ for each $i \in \{1,2, \dots, n\}$.
    We have $(1) \iff (2)$ directly from Definition~\ref{dfn:anchor}.
    Next, observe that the $i$-th row-sum of $A-\mathcal{E}$ is equal to $1-\mathcal{E}(i)$.
    Since $\mathcal{E}(i) > 0$ is equivalent to $1-\mathcal{E}(i)<1$, we obtain $(2) \iff (3)$.
\end{proof}

The following corollary identifies the anchors of any given row-substochastic matrix.

\begin{corollary} \label{cor:anchorsubstoch}
    Let $B$ be a row-substochastic matrix of size $n > 0$.
    A row/index $i \in \{1,2,\dots,n\}$ is an anchor of $B$ if and only if the $i$-th row-sum of $B$ is strictly less than $1$.
\end{corollary}

\begin{proof}
    Observe that there exists a unique pair of row-stochastic matrix $A$ and diagonal nonnegative matrix $\mathcal{E}$, both of size $n$, such that $B=A-\mathcal{E}$.
    The conclusion then immediately follows from $(1) \iff (3)$ in Proposition~\ref{prop:anchorAminE>=0}.
\end{proof}

By \Cref{cor:anchorsubstoch}, it is clear that any row-stochastic matrix $A$ does not have any anchor.
Outside the context of this paper, \Cref{cor:anchorsubstoch} can instead be used as the starting point to define anchors for row-substochastic matrices.

We now define \emph{condensely anchored} as follows:
\begin{definition} \label{dfn:condenseanchor}
    Let $A$ be a row-stochastic matrix and let $\mathcal{E}$ be a diagonal nonnegative matrix of the same size.
    We call $A-\mathcal{E}$ \textbf{condensely anchored} if every sink SCC of $G[A-\mathcal{E}]$ contains an anchor of $A-\mathcal{E}$.
\end{definition}

Note that Definition~\ref{dfn:condenseanchor} above parallels that of condensely aperiodic.
Clearly, if each row of $A-\mathcal{E}$ is an anchor, then $A-\mathcal{E}$ is condensely anchored.
On the other hand, if $A-\mathcal{E}$ does not have any anchor, then $A-\mathcal{E}$ is not condensely anchored.
For example, the row-stochastic matrix $A$, if written as $A-0$, is not condensely anchored.
Hence, this emphasizes that the definition of condensely anchored applies to $A-\mathcal{E}$ instead of $A$.

\begin{proposition} \label{prop:condanchor_sinkSCC}
    Let $A$ be a row-stochastic matrix and let $\mathcal{E}$ be a diagonal nonnegative matrix of the same size $n > 0$.
    Then $A-\mathcal{E}$ is condensely anchored if and only if for all $i \in \{1,2,\dots,n\}$, there exists a walk in the digraph $G[A-\mathcal{E}]$ from vertex $i$ to some anchor vertex of $A-\mathcal{E}$.
\end{proposition}

\begin{proof}
    Suppose that $A-\mathcal{E}$ is condensely anchored.
    By Definition~\ref{dfn:condenseanchor}, every sink SCC of $G[A-\mathcal{E}]$ contains an anchor of $A-\mathcal{E}$.
    Let $i$ be a vertex of $G[A-\mathcal{E}]$.
    Suppose that $i \in W$ where $W$ is an SCC of $G[A-\mathcal{E}]$.
    Then there is a walk starting at $i$ that will reach vertices in some sink SCC of $G[A-\mathcal{E}]$, possibly even $W$ itself.
    Since any sink SCC of $G[A-\mathcal{E}]$ contains an anchor, we conclude that there is a walk from $i$ to some anchor vertex of $A-\mathcal{E}$.

    Conversely, suppose that for all $i \in \{1,2,\dots,n\}$, there exists a walk in the digraph $G[A-\mathcal{E}]$ from vertex $i$ to some anchor vertex of $A-\mathcal{E}$.
    Let $W$ be a sink SCC of $G[A-\mathcal{E}]$ and let $i \in W$.
    By assumption, there is a walk from $i$ to some anchor vertex $j$ of $A-\mathcal{E}$.
    Since $W$ is a sink, we have that $j \in W$.
    Hence, by Definition~\ref{dfn:condenseanchor}, we conclude that $A-\mathcal{E}$ is condensely anchored.
\end{proof}

\begin{remark}
    Let $G = (V, E)$ be a digraph.
    Instead of involving anchors as defined in \Cref{dfn:anchor}, we can simply fix any subset $U \subseteq V$.
    Hence, we can jointly reformulate both Definition~\ref{dfn:condenseanchor} and Proposition~\ref{prop:condanchor_sinkSCC} as: Every sink SCC of $G$ contains a vertex from $U$ if and only if for all $i \in V$, there exists a walk in $G$ from $i$ to some vertex in $U$.
    This statement is inherently graph-theoretic, as it can be stated independently of the matrices $A$ and $\mathcal{E}$.
\end{remark}

There are two main parts in our approach to study the zero-convergence of $A-\mathcal{E}$.
First, if $A-\mathcal{E} \ge 0$ then $\rho(A-\mathcal{E}) < 1$ precisely when $A-\mathcal{E}$ is condensely anchored.
Although the condition $\rho(A-\mathcal{E}) < 1$ itself is relatively simple, the condensely-anchored property provides more meaningful interpretations in terms of dynamics.
In the second part, we study the case where the matrix $A-\mathcal{E}$ can contain some negative diagonal entries.
Assuming the learning rate of every agent is within the closed anchoring interval, condensely anchored also implies that $\rho(A-\mathcal{E}) < 1$.

\subsection{The case where \texorpdfstring{$A-\mathcal{E}$}{A-E} is nonnegative}
\label{subsec:AminEnonneg}

Let $A-\mathcal{E}$ be a nonnegative matrix.
Note that $0 \leq \mathcal{E}(i) \leq A(i, i)$ for each $i \in \{1,2, \dots, n\}$.
Since the $i$-th row-sum of $A-\mathcal{E}$ is equal to $1-\mathcal{E}(i)$, then $A-\mathcal{E}$ is also row substochastic.
It follows that $\rho(A-\mathcal{E}) \leq 1$.
Our main result in this section is that $\rho(A-\mathcal{E}) < 1$ precisely when $A-\mathcal{E}$ is condensely anchored.

First, we recall the \emph{index of contraction} from \cite{Azimzadeh19}.
Note that in \Cref{prop:condanchor_sinkSCC}, if the vertex $i$ itself is an anchor, then the desired walk, which is of length $0$, automatically exists.
We remark that the definition of walk in \cite[Definition 2.2]{Azimzadeh19} excludes what we regard as walks of length $0$.
Hence, with our notion of walks of length $0$, Definition~\ref{dfn:indexofcon} below is an equivalent reformulation of \cite[(2.1)]{Azimzadeh19}.

\begin{definition} \label{dfn:indexofcon}
    Let $B$ be a row-substochastic matrix of size $n>0$.
    Let $i \in \{1,2,\dots,n\}$ and let $\widehat{W}_i(B)$ be the set of walks $\omega$ in $G[B]$ such that $\omega$ starts at $i$ and ends at some vertex $j$, where the $j$-th row-sum of $B$ is strictly less than $1$.
    We define the \textbf{index of contraction} of $B$ as
    \[
    \widehat{\con}(B) \coloneqq \sup_{1 \leq i \leq n} \left\{ \inf_{\omega \in \widehat{W}_i(B)} \lvert \omega \rvert \right\}
    \]
    with the convention that $\inf \varnothing = \infty$.
\end{definition}

In Definition~\ref{dfn:indexofcon}, for each $i \in \{1,2,\dots,n\}$, we find $\omega \in \widehat{W}_i(B)$ of the least length, which we assign as $\infty$ if $\widehat{W}_i(B) = \varnothing$.
Taking the supremum over these minimum lengths, we obtain the index of contraction of $B$.

\begin{example} \label{ex:indexcon}
    Let
    \[
    B = \begin{pmatrix}
        1/3 & 2/3 & 0 & 0 \\
        0 & 0 & 1 & 0 \\
        3/4 & 0 & 0 & 1/4\\
        0 & 1/5 & 0 & 2/5\\
    \end{pmatrix}
    \]
    be a row-substochastic matrix.
    By \Cref{cor:anchorsubstoch}, note that only row $4$ of $B$ is an anchor.
    There exists a unique pair of row-stochastic matrix
    \[
    A = \begin{pmatrix}
        1/3 & 2/3 & 0 & 0 \\
        0 & 0 & 1 & 0 \\
        3/4 & 0 & 0 & 1/4\\
        0 & 1/5 & 0 & 4/5\\
    \end{pmatrix}
    \]
    and diagonal nonnegative matrix $\mathcal{E} = \diag(0,0,0,2/5)$ such that $B=A-\mathcal{E}$.
    The underlying digraphs $G[A]$ and $G[B]$ are identical, as depicted in \Cref{fig:indexcon}.
    However, we have $\widehat{\con}(A) = \infty$, while 
    \[
    \widehat{\con}(B) = \sup \{3,2,1,0\} = 3.
    \]
    One can check that $\norm{B}_{\infty} = \norm{B^2}_{\infty} = \norm{B^3}_{\infty} = 1$, while $\norm{B^4}_{\infty} < 1$, see \cite[Theorem 2.5]{Azimzadeh19}.
\end{example}

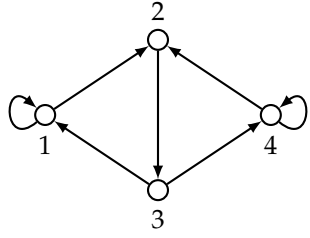
\begin{figure}
    \centering
    \begin{tikzpicture}[baseline=0]
        \node[circle, fill=white, draw=black, thick, scale=0.8, label=below:1] (v1) at (-3/2, 0) {};
        \node[circle, fill=white, draw=black, thick, scale=0.8, label=above:2] (v2) at (0, 1) {};
        \node[circle, fill=white, draw=black, thick, scale=0.8, label=below:3] (v3) at (0, -1) {};
        \node[circle, fill=white, draw=black, thick, scale=0.8, label=below:4] (v4) at (3/2, 0) {};

        \draw[-latex, thick] (v1) -- (v2);
        \draw[-latex, thick] (v2) -- (v3);
        \draw[-latex, thick] (v3) -- (v1);
        \draw[-latex, thick] (v3) -- (v4);
        \draw[-latex, thick] (v4) -- (v2);
        \draw[-latex, thick] (v1) to [out=-140,in=140, loop, style={min distance=6mm}] (v1);
        \draw[-latex, thick] (v4) to [out=-40,in=40, loop, style={min distance=6mm}] (v4);
    \end{tikzpicture}
    \caption{The underlying digraph of both matrices $A$ and $B$ in \Cref{ex:indexcon}.}
    \label{fig:indexcon}
\end{figure}

The next theorem is the main result of this section.

\begin{theorem}
\label{thm:condenseanchor}
    Let $A$ be a row-stochastic matrix of size $n > 0$, and let $\mathcal{E}$ be a diagonal nonnegative matrix of size $n$ such that $A-\mathcal{E} \ge 0$.
    Then, the following are equivalent:
    \begin{enumerate}[label=(\arabic*)]
        \item $A-\mathcal{E}$ is condensely anchored.
        \item For all $i \in \{1,2,\dots,n\}$, there exists a walk in the digraph $G[A-\mathcal{E}]$ from vertex $i$ to some anchor vertex of $A-\mathcal{E}$.
        \item $\widehat{\con}(A-\mathcal{E})$ is finite.
        \item $A-\mathcal{E}$ is zero-convergent.
    \end{enumerate}
\end{theorem}

\begin{proof}
    $(1) \iff (2)$ follows directly from \Cref{prop:condanchor_sinkSCC}.
    By \cite[Corollary 2.6]{Azimzadeh19} or \cite[Theorem A.1.6]{Zabaljauregui21}, we obtain $(3) \iff (4)$.
    We will now show that $(2) \iff (3)$.
    First, we know that $A-\mathcal{E}$ is row substochastic.
    From \Cref{dfn:indexofcon}, we have that $\widehat{W}_i(A-\mathcal{E})$ is the set of walks $\omega$ in $G[A-\mathcal{E}]$ such that $\omega$ starts at $i$ and ends at some vertex $j$, where the $j$-th row-sum of $A-\mathcal{E}$ is strictly less than $1$.
    By Proposition~\ref{prop:anchorAminE>=0}, we also have that a vertex $j$ is an anchor of $A-\mathcal{E}$ if and only if the $j$-th row-sum of $A-\mathcal{E}$ is strictly less than $1$.
    Hence, we obtain
    \[
    \omega \in \widehat{W}_i(A-\mathcal{E}) \iff \omega \text{ starts at } i \text{ and ends at some anchor vertex } j \text{ of } A-\mathcal{E}.
    \]
    It follows that $(2)$ is equivalent to
    \[
    \widehat{W}_i(A-\mathcal{E}) \neq \varnothing \text{ for all } i \in \{1,2,\dots,n\}.
    \]
    By \Cref{dfn:indexofcon}, we deduce that
    \[
    \widehat{W}_i(A-\mathcal{E}) \neq \varnothing \text{ for all } i \in \{1,2,\dots,n\} \iff \widehat{\con}(A-\mathcal{E}) \text{ is finite}.
    \]
    Therefore, $(2) \iff (3)$ is true, which completes the proof.
\end{proof}

\begin{remark}
    Let $B$ be a row-substochastic matrix of size $n>0$.
    We know that there exists a unique pair of row-stochastic matrix $A$ and diagonal nonnegative matrix $\mathcal{E}$, both of size $n$, such that $B=A-\mathcal{E}$.
    Hence, \Cref{thm:condenseanchor} can also be stated in terms of row-substochastic matrix $B$ instead of $A-\mathcal{E}$.
\end{remark}

It follows that a row-substochastic matrix $B$ is zero-convergent if and only if $B$ is condensely anchored.
Compare this equivalence with \Cref{thm:aperiodic}.
An alternative approach to derive this equivalence is by using the block matrix form \eqref{eq:blockform}.
As a corollary and an application of Theorem~\ref{thm:condenseanchor} in a more general framework, we state \Cref{cor:impairedlearner} in Section~\ref{sec:impaired_ave}.
Specifically, the condensely-anchored property is equivalent to the asymptotic decay of proper-substochastic DeGroot dynamics.

In Example~\ref{ex:condenseanchor}, we fix a row-stochastic matrix $A$, and we consider two diagonal nonnegative matrices $\mathcal{E}_1$ and $\mathcal{E}_2$.
For each matrix $A-\mathcal{E}_1$ and $A-\mathcal{E}_2$, there are four defective agents but only a single anchor.

\begin{example}
\label{ex:condenseanchor}
    Suppose that
    \[
    A = 
    \scalebox{0.82}{$\begin{pmatrix}
        1 & 0 & 0 & 0 & 0\\
        1/2 & 0 & 1/2 & 0 & 0\\
        0 & 0 & 0 & 1 & 0\\
        0 & 0 & 0 & 0 & 1\\
        0 & 1/4 & 0 & 0 & 3/4
    \end{pmatrix}.$}    
    \]
    Let $\mathcal{E}_1 = \diag(1/2, 0, 0, 0, 0)$ and $\mathcal{E}_2 = \diag(0, 0, 0, 0, 1/2)$, so we have
    \[
    A-\mathcal{E}_1 = \scalebox{0.82}{$\begin{pmatrix}
        1/2 & 0 & 0 & 0 & 0\\
        1/2 & 0 & 1/2 & 0 & 0\\
        0 & 0 & 0 & 1 & 0\\
        0 & 0 & 0 & 0 & 1\\
        0 & 1/4 & 0 & 0 & 3/4
    \end{pmatrix}$}
    \; \text{and} \;
    A-\mathcal{E}_2 = \scalebox{0.82}{$\begin{pmatrix}
        1 & 0 & 0 & 0 & 0\\
        1/2 & 0 & 1/2 & 0 & 0\\
        0 & 0 & 0 & 1 & 0\\
        0 & 0 & 0 & 0 & 1\\
        0 & 1/4 & 0 & 0 & 1/4
    \end{pmatrix}$}.
    \]
    The digraphs $G[A]$, $G[A-\mathcal{E}_1]$, and $G[A-\mathcal{E}_2]$ are all identical and condensely aperiodic, see \Cref{tab:3digraphs}.
    Moreover, $\{1\}$ is the only sink SCC of each of these digraphs.
    Vertex $1$ of $G[A-\mathcal{E}_1]$ is its sole anchor, while vertex $5$ of $G[A-\mathcal{E}_2]$ is its sole anchor.
    In Table~\ref{tab:3digraphs}, we can see that $A-\mathcal{E}_1$ is condensely anchored, while $A-\mathcal{E}_2$ is not condensely anchored.
    By Theorem~\ref{thm:condenseanchor}, we have $\displaystyle \lim_{t \to \infty} (A-\mathcal{E}_1)^t = 0$, while $A-\mathcal{E}_2$ is not zero-convergent.
    We find that
    \[
    \lim_{t \to \infty} A^t = \scalebox{0.85}{$\begin{pmatrix}
        1 & 0 & 0 & 0 & 0\\
        1 & 0 & 0 & 0 & 0\\
        1 & 0 & 0 & 0 & 0\\
        1 & 0 & 0 & 0 & 0\\
        1 & 0 & 0 & 0 & 0\\
    \end{pmatrix}$}
    \; \text{and} \;
    \lim_{t \to \infty} (A-\mathcal{E}_2)^t = \scalebox{0.85}{$\begin{pmatrix}
        1 & 0 & 0 & 0 & 0\\
        3/5 & 0 & 0 & 0 & 0\\
        1/5 & 0 & 0 & 0 & 0\\
        1/5 & 0 & 0 & 0 & 0\\
        1/5 & 0 & 0 & 0 & 0\\
    \end{pmatrix}$}.
    \]
\end{example}

\begin{table}[htbp]
    \centering
    \begin{adjustbox}{scale=0.80,center}
    \begin{tabular}{|c|c|c|c|c|}
    \hline
    Digraph & Condensation & Anchor & \parbox{2cm}{\centering Condensely anchored} & \parbox{2cm}{\centering Condensely aperiodic} \\ \hline 
       \begin{tikzpicture}[baseline=0]
        \node[circle, fill=white, draw=black, thick, scale=0.8, label=below:1] (v1) at (2, 0) {};
        \node[circle, fill=white, draw=black, thick, scale=0.8, label=below:2] (v2) at (0, 0) {};
        \node[circle, fill=white, draw=black, thick, scale=0.8, label=below:3] (v3) at (-1, -1) {};
        \node[circle, fill=white, draw=black, thick, scale=0.8, label=left:4] (v4) at (-2, 0) {};
        \node[circle, fill=white, draw=black, thick, scale=0.8, label=right:5] (v5) at (-1, 1) {};
        \node at (1, 1.1) {\large $G[A]$};

        \draw[-latex, thick] (v2) -- (v1);
        \draw[-latex, thick] (v2) -- (v3);
        \draw[-latex, thick] (v3) -- (v4);
        \draw[-latex, thick] (v4) -- (v5);
        \draw[-latex, thick] (v5) -- (v2);
        \draw[-latex, thick] (v1) to [out=-40,in=30, loop, style={min distance=6mm}] (v1);
        \draw[-latex, thick] (v5) to [out=130,in=60, loop, style={min distance=6mm}] (v5);
    \end{tikzpicture}
    & 
    \begin{tikzpicture}[baseline=0]
        \node[circle, fill=white, draw=black, thick, scale=0.8, label=below:\text{\{1\}}] (v1) at (2, 0) {};
        \node[circle, fill=white, draw=black, thick, scale=0.8, label=below:\text{\{2, 3, 4, 5\}}] (v2) at (0, 0) {};
        \node at (1, 1.1) {\large $(G[A])^\circ$};

        \draw[-latex, thick] (v2) -- (v1);
    \end{tikzpicture}
    & None & \textcolor{red}{\LARGE{\ding{55}}} & \textcolor{green}{\LARGE{\ding{51}}} \\ 
    \hline
    \begin{tikzpicture}[baseline=0]
        \node[circle, fill=blue, draw=black, thick, scale=0.8, label=below:1] (v1) at (2, 0) {};
        \node[circle, fill=white, draw=black, thick, scale=0.8, label=below:2] (v2) at (0, 0) {};
        \node[circle, fill=white, draw=black, thick, scale=0.8, label=below:3] (v3) at (-1, -1) {};
        \node[circle, fill=white, draw=black, thick, scale=0.8, label=left:4] (v4) at (-2, 0) {};
        \node[circle, fill=white, draw=black, thick, scale=0.8, label=right:5] (v5) at (-1, 1) {};
        \node at (1, 1.1) {\large $G[A-\mathcal{E}_1]$};

        \draw[-latex, thick] (v2) -- (v1);
        \draw[-latex, thick] (v2) -- (v3);
        \draw[-latex, thick] (v3) -- (v4);
        \draw[-latex, thick] (v4) -- (v5);
        \draw[-latex, thick] (v5) -- (v2);
        \draw[-latex, thick] (v1) to [out=-40,in=30, loop, style={min distance=6mm}] (v1);
        \draw[-latex, thick] (v5) to [out=130,in=60, loop, style={min distance=6mm}] (v5);
        \end{tikzpicture}
    & 
    \begin{tikzpicture}[baseline=0]
        \node[circle, fill=blue, draw=black, thick, scale=0.8, label=below:\text{\{1\}}] (v1) at (2, 0) {};
        \node[circle, fill=white, draw=black, thick, scale=0.8, label=below:\text{\{2, 3, 4, 5\}}] (v2) at (0, 0) {};
        \node at (1, 1.1) {\large $(G[A-\mathcal{E}_1])^\circ$};

        \draw[-latex, thick] (v2) -- (v1);
    \end{tikzpicture}
    & \multirow{2}{*}{\vspace{-12mm} Vertex 1}
    & \multirow{2}{*}{\vspace{-12mm} \textcolor{green}{\LARGE{\ding{51}}}} 
    & \multirow{2}{*}{\vspace{-12mm} \textcolor{green}{\LARGE{\ding{51}}}} \\
    \parbox{5cm}{\centering Any blue vertex is an anchor.} & \parbox{3.5cm}{\centering Any blue vertex is an SCC containing an anchor.\vspace{2mm}} & & & \\
    \hline
     \begin{tikzpicture}[baseline=0]
        \node[circle, fill=white, draw=black, thick, scale=0.8, label=below:1] (v1) at (2, 0) {};
        \node[circle, fill=white, draw=black, thick, scale=0.8, label=below:2] (v2) at (0, 0) {};
        \node[circle, fill=white, draw=black, thick, scale=0.8, label=below:3] (v3) at (-1, -1) {};
        \node[circle, fill=white, draw=black, thick, scale=0.8, label=left:4] (v4) at (-2, 0) {};
        \node[circle, fill=blue, draw=black, thick, scale=0.8, label=right:5] (v5) at (-1, 1) {};
        \node at (1, 1.1) {\large $G[A-\mathcal{E}_2]$};

        \draw[-latex, thick] (v2) -- (v1);
        \draw[-latex, thick] (v2) -- (v3);
        \draw[-latex, thick] (v3) -- (v4);
        \draw[-latex, thick] (v4) -- (v5);
        \draw[-latex, thick] (v5) -- (v2);
        \draw[-latex, thick] (v1) to [out=-40,in=30, loop, style={min distance=6mm}] (v1);
        \draw[-latex, thick] (v5) to [out=130,in=60, loop, style={min distance=6mm}] (v5);
    \end{tikzpicture} & 
    \begin{tikzpicture}[baseline=0]
        \node[circle, fill=white, draw=black, thick, scale=0.8, label=below:\text{\{1\}}] (v1) at (2, 0) {};
        \node[circle, fill=blue, draw=black, thick, scale=0.8, label=below:\text{\{2, 3, 4, 5\}}] (v2) at (0, 0) {};
        \node at (1, 1.1) {\large $(G[A-\mathcal{E}_2])^\circ$};

        \draw[-latex, thick] (v2) -- (v1);
    \end{tikzpicture}
    & \multirow{2}{*}{\vspace{-12mm} Vertex 5}
    & \multirow{2}{*}{\vspace{-12mm} \textcolor{red}{\LARGE{\ding{55}}}} 
    & \multirow{2}{*}{\vspace{-12mm} \textcolor{green}{\LARGE{\ding{51}}}} \\
    \parbox{5cm}{\centering Any blue vertex is an anchor.} & \parbox{3.5cm}{\centering Any blue vertex is an SCC containing an anchor.\vspace{2mm}} & & & \\
    \hline
    \end{tabular}
    \end{adjustbox}
    \caption{The digraphs $G[A]$, $G[A-\mathcal{E}_1]$, and $G[A-\mathcal{E}_2]$ from Example~\ref{ex:condenseanchor}.}
    \label{tab:3digraphs}
\end{table}

\subsection{The case where \texorpdfstring{$A-\mathcal{E}$}{A-E} can contain negative diagonal entries}
\label{subsec:AminEnegdiag}

In this section, we consider matrices $A-
\mathcal{E}$ that contain some negative diagonal entries.
Building upon Theorem~\ref{thm:condenseanchor}, our next result still involves the condensely-anchored property.

\begin{theorem}
\label{thm:extendcondenseanchor}
    Let $A$ be a row-stochastic matrix of size $n > 0$ and let $\mathcal{E}$ be a diagonal nonnegative matrix of size $n$.
    Suppose that $0 \leq \mathcal{E}(i) \leq 2A(i,i)$ for all $i \in \{1,2,\dots,n\}$.
    If $A-\mathcal{E}$ is condensely anchored then $A-\mathcal{E}$ is zero-convergent.
\end{theorem}

\begin{proof}
    Suppose that $A-\mathcal{E}$ is condensely anchored and consider the matrix $\lvert A-\mathcal{E} \rvert \geq 0$.
    Note that the matrices $A-\mathcal{E}$ and $\lvert A-\mathcal{E} \rvert$ are almost identical, except possibly their diagonal entries.
    For all $i \in \{1,2,\dots,n\}$, we have that
    \begin{equation*}
        0 \leq \mathcal{E}(i) \leq 2A(i,i) \iff \lvert A(i,i)-\mathcal{E}(i) \rvert \leq A(i,i).
    \end{equation*}
    This implies that each row-sum of $\lvert A-\mathcal{E} \rvert$ is at most $1$, and hence, $\lvert A-\mathcal{E} \rvert$ is row substochastic.
    Then, there exists a row-stochastic matrix $\bar{A}$ and a diagonal nonnegative matrix $\bar{\mathcal{E}}$ such that $\lvert A-\mathcal{E} \rvert = \bar{A}-\bar{\mathcal{E}} \geq 0$.
    If $i \in \{1,2,\dots,n\}$ is an anchor of $A-\mathcal{E}$, then
    \begin{equation*}
        0 < \mathcal{E}(i) < 2A(i,i) \iff \lvert A(i,i)-\mathcal{E}(i) \rvert < A(i,i),
    \end{equation*}
    where $\lvert A(i,i)-\mathcal{E}(i) \rvert < A(i,i)$ means that the $i$-th row-sum of $\lvert A-\mathcal{E} \rvert = \bar{A}-\bar{\mathcal{E}}$ is less than $1$.
    By Proposition~\ref{prop:anchorAminE>=0}, we deduce that $i \in \{1,2,\dots,n\}$ is an anchor of $\bar{A}-\bar{\mathcal{E}}$ if and only if $i$ is an anchor of $A-\mathcal{E}$.
    The digraphs $G[A-\mathcal{E}]$ and $G[\bar{A}-\bar{\mathcal{E}}]$ are also identical.
    Hence, since $A-\mathcal{E}$ is condensely anchored, then $\bar{A}-\bar{\mathcal{E}}$ is also condensely anchored.
    By Theorem~\ref{thm:condenseanchor}, we have that $\bar{A}-\bar{\mathcal{E}}$ is zero-convergent.
    By Theorem~\ref{thm:zeroconv}, we have $\rho(\bar{A}-\bar{\mathcal{E}}) < 1$.
    Since
    \[
    \rho(A-\mathcal{E}) \leq \rho(\lvert A-\mathcal{E} \rvert) = \rho(\bar{A}-\bar{\mathcal{E}}) < 1,
    \]
    then $A-\mathcal{E}$ is zero-convergent, also by Theorem~\ref{thm:zeroconv}.
\end{proof}

The assumption that $0 \leq \mathcal{E}(i) \leq 2A(i,i)$ for all $i \in \{1,2,\dots,n\}$ is crucial for \Cref{thm:extendcondenseanchor}.
For example, let
\begin{equation*}
    A = \begin{pmatrix}
    1/2 & 1/2 \\
    3/4 & 1/4
\end{pmatrix} \; \text{and} \;\, \mathcal{E} = \begin{pmatrix}
        3/4 & 0\\
        0 & 3
    \end{pmatrix} \implies A-\mathcal{E} = \begin{pmatrix}
     -1/4 & 1/2 \\
     3/4 & -11/4
\end{pmatrix}.
\end{equation*}
Since $0 < \mathcal{E}(1) < 2A(1,1)$ and $\mathcal{E}(2) \ge 2A(2,2)$, the first row of $A-\mathcal{E}$ is an anchor, while the second row of $A-\mathcal{E}$ is a non-anchor.
Hence, $A-\mathcal{E}$ is condensely anchored but $\rho(A-\mathcal{E})>1$, and thus, $A-\mathcal{E}$ is not zero-convergent.

Compared with \Cref{thm:condenseanchor}, in \Cref{thm:extendcondenseanchor}, we lose equivalence.
For example, let
\begin{equation} \label{eq:A-E}
    A = \begin{pmatrix}
    3/5 & 2/5 \\
    2/3 & 1/3
\end{pmatrix} \; \text{and} \;\, \mathcal{E} = \begin{pmatrix}
        0 & 0\\
        0 & 2/3
    \end{pmatrix} \implies A-\mathcal{E} = \begin{pmatrix}
    3/5 & 2/5 \\
    2/3 & -1/3
\end{pmatrix}.
\end{equation}
Since $\mathcal{E}(1)=0$ and $\mathcal{E}(2) = 2 A(2,2)$, $A-\mathcal{E}$ has no anchor and thus, it cannot be condensely anchored.
However, we compute that $\rho(A-\mathcal{E}) < 0.83$, which means that $A-\mathcal{E}$ is zero-convergent.

Using \Cref{prop:conv_implies_consensus}, the following corollary of \Cref{thm:extendcondenseanchor} is straightforward.

\begin{corollary} \label{cor:condanchor}
    Let $A$ be a row-stochastic matrix of size $n > 0$ and let $\mathcal{E}$ be a diagonal nonnegative matrix of size $n$.
    Suppose that $0 \leq \mathcal{E}(i) \leq 2A(i,i)$ for all $i \in \{1,2,\dots,n\}$.
    If $A-\mathcal{E}$ is condensely anchored then all agents eventually reach consensus, i.e., $\displaystyle \lim_{t \to \infty} X_t = \ms$.
\end{corollary}

\begin{remark}
    \Cref{cor:condanchor} generalizes \Cref{prop:positivelearn}, which requires all agents to be anchors.
    Consensus can still be reached even if there is only one anchor as long as the matrix $A-\mathcal{E}$ is condensely anchored.
    We can interpret this situation as follows: Suppose that the learning rate of each agent is inside the closed anchoring interval.
    If for every agent, there is a flow of information to any anchor, then consensus will be eventually reached.
    Equivalently, if each strongly communicating group of agents that leaks no information contains an anchor, then consensus will be eventually reached.
\end{remark}

As we have seen above, the matrix $A-\mathcal{E}$ in \eqref{eq:A-E} is not condensely anchored but consensus is still reached eventually.
It is worth pointing out a subtlety for the case $A-\mathcal{E} \ge 0$.
By \Cref{thm:condenseanchor}, the condensely-anchored property is equivalent to consensus reached, regardless of $X_0-\ms$.
However, as previously mentioned, even if
\[
A-\mathcal{E} \ge 0 \text{ is not condensely anchored} \implies A-\mathcal{E} \text{ is not zero-convergent},
\]
we can still have $\lim_{t \to \infty} X_t = \ms$, though it will depend on $X_0-\ms$.
Since our current main interests are predicting consensus independent of $X_0-\ms$ (i.e., $\rho(A-\mathcal{E})<1$), even though \Cref{thm:condenseanchor} classifies condensely anchored for $A-\mathcal{E} \ge 0$, \Cref{cor:condanchor} is more versatile for the context of this paper.

Attempts at applying the mixed operator norm from \Cref{sec:timevarmixed} to the time-invariant case leads to \Cref{prop:highratiolearning} below.
We focus on the possibility that $\rho(A-\mathcal{E})<1$ even if there exists at least an $i$ where $\mathcal{E}(i)$ is far greater than $2A(i,i)$.
In \Cref{prop:highratiolearning}, note that $\mathcal{E}(1) = (n^2-n+1) A(1,1)$.
Though it originates from experimenting with mixed operator norm, \Cref{prop:highratiolearning} can simply be derived with standard tools of linear algebra.

\begin{proposition} \label{prop:highratiolearning}
Assume $n \geq 2$ is an integer.
Let
\begin{equation*}
    A = \begin{pmatrix}
    \frac{1}{n^2} & \frac{n+1}{n^2} \mathbf{1}_{n-1}^\top \\
    \frac{1}{n} \mathbf{1}_{n-1} & \frac{n-1}{n} I_{n-1}
    \end{pmatrix} \; \text{and} \;\; \mathcal{E} = \begin{pmatrix}
        \frac{n^2-n+1}{n^2} & 0\\
        0 & \frac{n-2}{n} I_{n-1}
    \end{pmatrix}.
\end{equation*}
Then $A-\mathcal{E}$ is zero-convergent.
\end{proposition}

\begin{proof}
    Observe that
    \[
    A-\mathcal{E}-I_n/n = \begin{pmatrix}
    -1 & \frac{n+1}{n^2} \mathbf{1}_{n-1}^\top \\
    \frac{1}{n} \mathbf{1}_{n-1} & 0
    \end{pmatrix}.
    \]
    Note that $A-\mathcal{E}-I_n/n$ has rank $2$, i.e., it has eigenvalue $0$ of multiplicity $n-2$.
    Let $\lambda_1$ and $\lambda_2$ be the remaining two eigenvalues of $A-\mathcal{E}-I_n/n$.
    Since $\tr(A-\mathcal{E}-I_n/n) = -1$ and $\tr((A-\mathcal{E}-I_n/n)^2) = 1 + 2/n - 2/n^3$, we have
    \begin{align*}
        \lambda_1 + \lambda_2 &= -1,\\
        \lambda_1^2 +\lambda_2^2 &= 1 + \frac{2}{n} - \frac{2}{n^3}.
    \end{align*}
    It follows that
    \[
    \{ \lambda_1, \lambda_2 \} = \left\{ -\frac{1}{2} \pm \sqrt{\frac{1}{4} + \frac{1}{n} - \frac{1}{n^3}} \right\}.
    \]
    Thus, the eigenvalues of $A-\mathcal{E}$ are $1/n$ of multiplicity $n-2$, together with
    \[
    \frac{1}{n} - \frac{1}{2} \pm \sqrt{\frac{1}{4} + \frac{1}{n} - \frac{1}{n^3}}.
    \]
    It is clear that $0 < 1/n < 1$ and note that
    \[
    \sqrt{\frac{1}{4} + \frac{1}{n} - \frac{1}{n^3}} < \sqrt{\frac{1}{4} + \frac{1}{n} + \frac{1}{n^2}} = \frac{1}{2}+\frac{1}{n}.
    \]
    Hence, we obtain
    \begin{align*}
        \left\lvert \frac{1}{n} - \frac{1}{2} \pm \sqrt{\frac{1}{4} + \frac{1}{n} - \frac{1}{n^3}} \right\rvert &\le \left\lvert \frac{1}{n} - \frac{1}{2} \right\rvert + \sqrt{\frac{1}{4} + \frac{1}{n} - \frac{1}{n^3}} = \frac{1}{2} - \frac{1}{n} + \sqrt{\frac{1}{4} + \frac{1}{n} - \frac{1}{n^3}} \\
        &< \frac{1}{2}-\frac{1}{n}+\frac{1}{2}+\frac{1}{n} = 1.
    \end{align*}
    Therefore, we conclude that $\rho(A-\mathcal{E})<1$ and by \Cref{thm:zeroconv}, $A-\mathcal{E}$ is zero-convergent.
\end{proof}

\Cref{ex:extreme2} below provides an even more extreme scenario.

\begin{example} \label{ex:extreme2}
Let
\begin{equation*}
    A = \begin{pmatrix}
    1/2^{100} & (2^{100}-1)/2^{100} \\
    1/2 & 1/2
    \end{pmatrix} \; \text{and} \;\, \mathcal{E} = \begin{pmatrix}
        1/2 & 0\\
        0 & 1/3
    \end{pmatrix}.
\end{equation*}
Note that $\mathcal{E}(1)=2^{99} A(1,1)$ and $\rho(A-\mathcal{E}) < 0.95$.
Hence, $A-\mathcal{E}$ is still zero-convergent.
\end{example}

Another scenario below is when the learning rates of numerous agents can be greater than but are also independent of the values of their self-beliefs.

\begin{example} \label{ex:indep}
Let $A$ be a row-stochastic matrix of size $n \ge 2$ and let $\mathcal{E}$ be a diagonal nonnegative matrix of size $n$.
Let $A(1,1)=1$ and for all $i \in \{2,3,\dots,n\}$, let $A(i, i-1) = 1$.
Then, $A-\mathcal{E}$ is zero-convergent if and only if $0 < \mathcal{E}(1) <2$ and $0 \le \mathcal{E}(i) < 1$ for all $i \in \{2,3,\dots,n\}$.
This is clear since 
\[
\rho(A-\mathcal{E}) = \max_{1 \leq i \leq n} \lvert (A-\mathcal{E})(i,i) \rvert.
\]
\end{example}

\Cref{prop:highratiolearning}, \Cref{ex:extreme2}, and \Cref{ex:indep} signify the complexity of classifying $A-\mathcal{E}$, where $A-\mathcal{E}$ can contain some negative diagonal entries, such that $\rho(A-\mathcal{E})<1$.
There does not seem to exist a concise enough condition that resembles the condensely-anchored property for the case $A-\mathcal{E} \ge 0$.

Determining when $\rho(A-\mathcal{E}) \ge 1$ motivates the following problem, which may be of independent interest.

\begin{problem} \label{prob:genconstantbound}
    Let $A$ be a row-stochastic matrix of size $n > 0$ and let $\mathcal{E}$ be a diagonal nonnegative matrix of size $n$.
    Suppose that there exists an $i \in \{1,2,\dots,n\}$ such that $\mathcal{E}(i) \geq 3$.
    Is it true that $\rho(A-\mathcal{E}) \geq 1$?
\end{problem}

Partial results and progress on this problem and its variations are delegated to \Cref{appendix}.

\subsection{Impaired averaging} \label{sec:impaired_ave}
Having established the algebraic and graph properties of condensely anchored networks, we now contextualize this mechanism within the broader literature on opinion dynamics. Within the framework of the pure DeGroot model \cite{degroot1974reaching}, agents engage in the iterative process of averaging the beliefs held by their neighbours. This refined model has emerged as a fundamental tool in the study of social learning \cite{cao2008reaching,chazelle2011total,fagnani2009average,gao2025opinion,sobel2000economists}. Despite its inherent simplicity, the DeGroot model exhibits fragility: the system's dynamics lack robustness, and the inclusion of even a solitary erroneous agent, who fails to compute averages accurately, has the potential to skew the consensus towards any arbitrary value \cite{acemouglu2013opinion,amir2025granular,li2025upper,rabinovich2021erratic, shi2025consensus,sridhar2023mean,zelazo2015robustness}.

An important extension of the classical DeGroot model is the Friedkin--Johnsen (FJ) model 
\cite{friedkin1990social,friedkin2011social}, which incorporates both social influence 
and agents' intrinsic predispositions. The dynamics are given by
\begin{equation}\label{eq:FJ-dynamics}
X_{t+1} = \Lambda A X_t \;+\; (I-\Lambda)X_0,
\end{equation}
where $\Lambda = \diag(\lambda_1,\dots,\lambda_n)$ with $0\le \lambda_i \le 1$ 
    encodes each agent’s \emph{susceptibility to social influence},  $I-\Lambda$ encodes the complementary \emph{stubbornness} or weight on 
    initial predispositions $X_0$. For $\lambda_i=1$, agent $i$ is fully susceptible and updates exactly as in the 
DeGroot model. For $\lambda_i=0$, agent $i$ is completely stubborn, anchoring 
their opinion forever at $X_t(i)$. For intermediate values $0<\lambda_i<1$, 
agents blend social influence with anchoring to their original predisposition. Unlike the pure DeGroot model, which generically converges to consensus under mild 
connectivity assumptions, the Friedkin--Johnsen dynamics converge to a nontrivial 
compromise equilibrium given by
\begin{equation}\label{eq:FJ-equilibrium}
    X^\star \;=\; (I - \Lambda A)^{-1}(I - \Lambda) X_0.
\end{equation}
This fixed point reflects a balance between interpersonal influence and 
anchoring to initial opinions. In particular, disagreement can persist in 
equilibrium, making the model suitable for explaining opinion diversity in social groups \cite{friedkin1990social, BallottaFJ2024}. The modification is an \emph{external anchoring} to initial opinions, so the equilibrium is a nontrivial compromise between $X_0$ and interpersonal influence. Persistent disagreement is the generic outcome.

While the Friedkin-Johnsen model produces substochastic dynamics through 
\emph{intentional anchoring} to initial opinions, we now consider a fundamentally 
different form of substochasticity arising from \emph{unintentional dissipation}. 
Formally, if some rows of the averaging matrix $A$ do not sum to one, the system 
becomes \emph{proper substochastic}. This corresponds to a loss of mass in the update, 
as though information is dissipating at each step. Concrete interpretations include: An agent discounts available signals so that $ \sum_{j=1}^n A_t(i,j) < 1$ for some $i$. An agent forgets part of its state when forming $X_{t+1}$ from $X_t$. And finally, external dissipation, opinion strength leaks out of the system. In the absence of stubborn agents or exogenous learning inputs, such impairments 
force the state toward the zero vector, which becomes the unique fixed point of 
the dynamics. This contrasts sharply with FJ, where the missing weight is preserved 
through anchoring to $X_0$, yielding persistent disagreement at a nontrivial equilibrium. 
In faulty averaging, the missing weight simply vanishes: consensus through extinction 
rather than compromise. In this sense, impaired averaging can substitute for perturbations or feedback 
control: if a row-sum is less than $1$, then the spectral radius drops below $1$, which ensures asymptotic decay.
The following corollary of Theorem~\ref{thm:condenseanchor} makes this precise.

\begin{corollary}\label{cor:impairedlearner}
     Let $A$ be a proper-substochastic matrix of positive size.
     Then $A$ is zero-convergent if and only if $A$ is condensely anchored.
\end{corollary}

\Cref{fig:substochastic} illustrates how condensely anchored, and condensely aperiodic, fits into substochasticity.
\Cref{cor:impairedlearner} highlights that proper-substochastic DeGroot dynamics deserve study in their own right as models of imperfect information processing, rather than merely as technical artifacts arising from stubborn or adversarial agents. An open direction is to consider the opposite case of ``zealots" whose row-sums exceed one, which may lead to opinion divergence to infinity.

\begin{figure}
\centering

\begin{tikzpicture}[
    scale=0.85, transform shape, 
    every node/.style={font=\small},
    lbl/.style={font=\small, align=center},
    properColor/.style={fill=teal!10},
    stochColor/.style={fill=gray!20},
    aperiodicColor/.style={fill=blue!50, fill opacity=0.3},
    anchoredColor/.style={fill=purple!50, fill opacity=0.3}
]

\fill[properColor] (0,0) circle (3cm);

\begin{scope}
    \clip (0,0) circle (3cm);
    \fill[stochColor]
        (0.2, 3.5) .. controls (0.0, 1.5) and (0.0, -1.5) .. (0.2, -3.5)
        -- (4,-3.5) -- (4,3.5) -- cycle;
\end{scope}


\begin{scope}
    \clip (0,0) circle (3cm);
    \fill[blue!40, fill opacity=0.3] 
        (-3.5, 3.5) -- (3.5, 3.5) -- (3.5, 1.5) 
        .. controls (0.8, 1.0) and (-0.8, 1.0) .. (-3.5, 1.5) -- cycle;
    \fill[pattern=north west lines, pattern color=blue!60!black, opacity=0.4]
        (-3.5, 3.5) -- (3.5, 3.5) -- (3.5, 1.5) 
        .. controls (0.8, 1.0) and (-0.8, 1.0) .. (-3.5, 1.5) -- cycle;
    
    \draw[blue!70!black, thin] (-3.5, 1.5) .. controls (-0.8, 1.0) and (0.8, 1.0) .. (3.5, 1.5);
\end{scope}

\begin{scope}
    \clip (0,0) circle (3cm);
    \fill[purple!40, fill opacity=0.3]
        ({3*cos(130)},{3*sin(130)}) 
        arc[start angle=130, end angle=230, radius=3]
        .. controls (-1.5, -0.5) and (-1.5, 1.8) .. 
        ({3*cos(130)},{3*sin(130)}) -- cycle;
    \fill[pattern=grid, pattern color=purple!60!black, opacity=0.4]
        ({3*cos(130)},{3*sin(130)}) 
        arc[start angle=130, end angle=230, radius=3]
        .. controls (-1.5, -0.5) and (-1.5, 1.8) .. 
        ({3*cos(130)},{3*sin(130)}) -- cycle;

    \draw[purple!70!black, thin]
        ({3*cos(130)},{3*sin(130)})
        .. controls (-1.5, 1.8) and (-1.5, -0.5) ..
        ({3*cos(230)},{3*sin(230)});
\end{scope}


\draw[thick, black!80] (0.2, 3) .. controls (0.0, 1.5) and (0.0, -1.5) .. (0.2, -3);

\draw[black, thick] (0,0) circle (3cm);

\node[font=\normalsize\bfseries] at (0,3.6) {Row-substochastic matrices};

\def\bracespace{0.1}
\draw[decorate, decoration={brace, amplitude=7pt, mirror}] 
    (-3.0,-3.6) -- ({0.15-\bracespace},-3.6) 
    node[midway, below=7pt, lbl] {Proper substochastic\\$\rho(A)\le 1$};

\draw[decorate, decoration={brace, amplitude=7pt, mirror}] 
    ({0.1+\bracespace},-3.6) -- (3.0,-3.6) 
    node[midway, below=7pt, lbl] {Row stochastic\\$\rho(A)=1$};

\begin{scope}[shift={(3.4,1.0)}]
    \draw[rounded corners, fill=white, draw=gray!50] (0,0) rectangle (6.3,-2.5);
    
    \fill[blue!40, fill opacity=0.3] (0.3,-0.6) rectangle (1.0,-1.1);
    \fill[pattern=north west lines, pattern color=blue!60!black, opacity=0.4] (0.3,-0.6) rectangle (1.0,-1.1);
    \draw[blue!70!black, thin] (0.3,-0.6) rectangle (1.0,-1.1);
    \node[anchor=west] at (1.3,-0.85) {Condensely aperiodic};
    
    \fill[purple!40, fill opacity=0.3] (0.3,-1.4) rectangle (1.0,-1.9);
    \fill[pattern=grid, pattern color=purple!60!black, opacity=0.4] (0.3,-1.4) rectangle (1.0,-1.9);
    \draw[purple!70!black, thin] (0.3,-1.4) rectangle (1.0,-1.9);
    \node[anchor=west] at (1.3,-1.65) {Condensely anchored ($\rho(A)<1$)};
\end{scope}

\end{tikzpicture}

\caption{A schematic diagram depicting the relationships among various subsets of row-substochastic matrices.}
\label{fig:substochastic}
\end{figure}
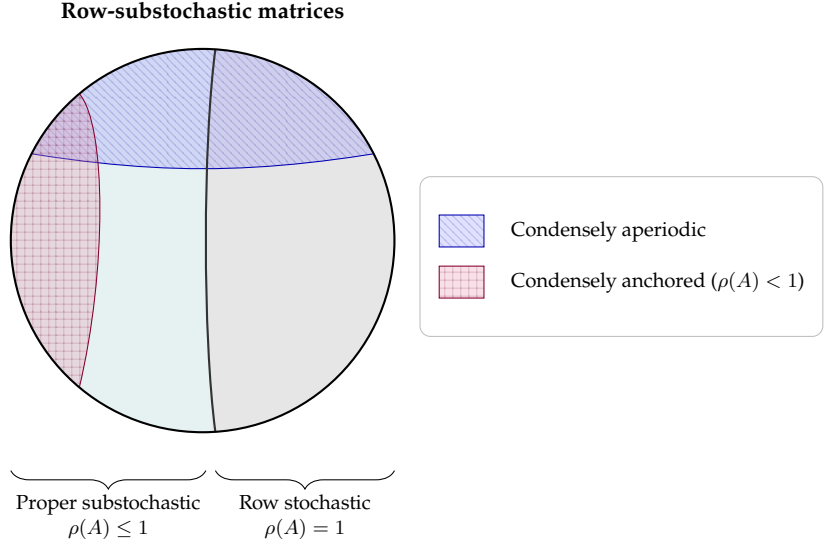

\section{Time-variant case}\label{sec:timvar}
We now turn to time-varying dynamics \cite{Jadbabiecoord2003,Tsitsiklis86dist}. First, recall the update rule from \Cref{eq:average+learn}:
\begin{equation}
    X_{t+1} - \ms = B_t(X_t - \ms),
\end{equation}
where $B_t \coloneqq A_t - \mathcal{E}_t$ for integers $t \geq 0$.

In Section~\ref{sec:timeinvar}, the spectral radius dictated convergence. Here, it fails. A sequence of matrices can satisfy $\rho(B_t) < 1$ at every step, yet their infinite product diverges. Spectral methods are no longer enough. We need stronger tools.

We break the analysis into two regimes:
\begin{enumerate}
    \item \emph{Vanishing Learning:} What happens when the learning rate $\mathcal{E}_t(i)$ decays to zero? Does the network lose the ground truth? We establish a ``minimal persistence'' condition on the decay rates alone. If the learning persists just enough, convergence holds.
    \item \emph{Mixed Operator Norms:} We build a new framework extracting two-step contraction from the interplay of alternating operator norms ($\ell_1 \to \ell_\infty$ and $\ell_\infty \to \ell_1$). This guarantees contractivity in intermittent environments exactly where single-norm and spectral bounds fail.
\end{enumerate}

\subsection{Vanishing learning rates}
In the earlier regime $0<\mathcal{E}_t(i)<2 A_t(i,i)$ but now we add the possibility that learning ability is vanishing. 
Each agent eventually has positive self-belief, however, no additional structure is required on the interaction matrix $A$. Condition (A3) ($\sum_{t=0}^{\infty} \min_i \mathcal{E}_t(i) = \infty$) is analogous to the slow cooling condition in simulated annealing. As established by \cite{geman1984stochastic, hajek1988cooling}, this condition ensures sufficient 'total cooling' to overcome any finite energy barrier and reach the global optimum. 
In our context, with effective temperature $T_t= 1/ \mathcal{E}_{\min,t}$ condition (A3) guarantees that $\displaystyle \int_0^{\infty} \mathcal{E}_{\min,t} dt = \infty$, allowing the system to escape any metastable state (wrong consensus) and reach consensus.


\begin{theorem}[Consensus with Vanishing Rates]\label{thm:main_vanish}
Consider the dynamics \eqref{eq:average+learn}.
Suppose that the learning rates $\{\mathcal{E}_t(i)\}_{t=0}^{\infty}$ satisfy the following conditions for all $i \in \{1,\ldots,n\}$:
\begin{enumerate}
    \item[(A1)] \textit{Eventual boundedness}: There exists $T_1 \in \N$ such that for all $t \geq T_1$:
    $$0 \leq \mathcal{E}_t(i) \leq A_{t}(i,i).$$ 
    \item[(A2)] \textit{Vanishing}: $\lim_{t \to \infty} \mathcal{E}_t(i) = 0$.
    \item[(A3)] \textit{Minimum persistence}: $\sum_{t=0}^{\infty} \min_{i=1,\ldots,n} \mathcal{E}_t(i) = \infty$.
\end{enumerate}
Then all agents reach consensus:
\[
\lim_{t \to \infty} X_t = \ms.
\]
\end{theorem}
\begin{proof}
For any $t \geq T_1$, by assumption (A1), we have $0 \leq \mathcal{E}_t(i) \leq A_t(i,i)$.  Therefore, $|A_t(i,i)- \mathcal{E}_t(i)| = A_t(i,i)- \mathcal{E}_t(i)$ for all $t \geq T_1$. 

For the infinity norm of $Y_{t+1}$ when $t \geq T_1$:
\begin{align*}
    |Y_{t+1}(i)| &= \left|\sum_{j=1}^{n} B_t(i,j)\,Y_t(j)\right| \leq \sum_{j=1}^{n} |B_t(i,j)|\,|Y_t(j)| \leq \|Y_t\|_{\infty} \sum_{j=1}^{n} |B_t(i,j)|
\end{align*}

Calculating the sum of the row for $t \geq T_1$: \begin{align*}
\sum_{j=1}^{n} |B_t(i,j)| &= |A_t(i,i) - \mathcal{E}_t(i)| + \sum_{j \neq i} |A_t(i,j)|\\
&= (A_t(i,i) - \mathcal{E}_t(i)) + \sum_{j \neq i} A_t(i,j) \quad \text{(since } A_t(i,j) \geq 0 \text{ and } A_t(i,i) \geq \mathcal{E}_t(i) \text{)}\\
&= A_t(i,i) - \mathcal{E}_t(i) + (1 - A_t(i,i)) \quad \text{(row-stochasticity of } A\text{)}\\
&= 1 - \mathcal{E}_t(i).
\end{align*}

Define $\beta_t \coloneqq \max_{i=1,\ldots,n}(1 - \mathcal{E}_t(i))$ for $t \geq T_1$. We have:
$$\beta_t = 1 - \min_{i=1,\ldots,n} \mathcal{E}_t(i) = 1 - \mathcal{E}_{\min,t} \leq 1$$
where $\mathcal{E}_{\min,t} = \min_{i} \mathcal{E}_t(i) \geq 0$ for $t \geq T_1$. Note that by definition, $\beta_t = \max_i(1-\mathcal{E}_t(i)) = 1-\min_i \mathcal{E}_t(i)$. The dynamical system then becomes, for $t \geq T_1$:
$$\|Y_{t+1}\|_{\infty} \leq \beta_t \|Y_t\|_{\infty}.$$
For any $t > T_1$, we can write:
\begin{align*}
    \|Y_t\|_{\infty} &\leq \beta_{t-1} \|Y_{t-1}\|_{\infty}
    \leq \prod_{k=T_1}^{t-1} \beta_k \cdot \|Y_{T_1}\|_{\infty} = \prod_{k=T_1}^{t-1} (1 - \mathcal{E}_{\min,k}) \cdot \|Y_{T_1}\|_{\infty}.
\end{align*}
We need to show that $\prod_{k=T_1}^{\infty} (1 -\mathcal{E}_{\min,t}) = 0$.
Note that $\ln(1-x) \leq -x$ for all $0 \leq x \le 1$.

\begin{align*}
\ln\left(\prod_{k=T_1}^{\infty} (1 - \mathcal{E}_{\min,k})\right) = \sum_{k=T_1}^{\infty} \ln(1 - \mathcal{E}_{\min,k})
\leq -\sum_{k=T_1}^{\infty} \mathcal{E}_{\min,k}.
\end{align*}

By assumption (A3), we have:
\[
\sum_{t=0}^{\infty} \min_{i=1,\ldots,n} \mathcal{E}_t(i) = \infty.
\]

The tail sum must likewise diverge because we are eliminating a finite number of terms (from $t=0$ to $t=T_1-1$):
\[
\sum_{t=T_1}^{\infty} \mathcal{E}_{\min,t} = \sum_{t=T_1}^{\infty} \min_{i=1,\ldots,n} \mathcal{E}_t(i) = \infty.
\]

Therefore, we obtain the final step
\[
\sum_{k=T_1}^{\infty} \ln(1 - \mathcal{E}_{\min,k}) \leq -\sum_{k=T_1}^{\infty} \mathcal{E}_{\min,k} = -\infty \implies \prod_{k=T_1}^{\infty} (1 - \mathcal{E}_{\min,k})=0.
\]
And this allows us to have:
\[\lim_{t \to \infty} \|Y_t\|_{\infty} = 0\] meaning $\displaystyle \lim_{t \to \infty} X_t = \ms. $
\end{proof}

Positive self-belief of the agents $A_t(i,i)>0$ was not required at all times, indeed this is milder than requiring the matrix $A_t$ to be primitive or aperiodic and irreducible: full connectedness is not required, allowing considerable freedom in the type of network topologies possible. There could be finite periods where $A_t(i,i)=0$. Furthermore, when $A_t(i,i)=0$, the learning ability is also zero: $\mathcal{E}_t(i)=0$. Agents have the potential to achieve the ground truth as long as they are acquiring knowledge and their decline in learning capacity is measured. If this deterioration is too rapid, they will fail to attain $\ms$.
\begin{remark}
We provide some examples of learning rates.
\begin{enumerate}
    \item Suppose the learning rates are only positive or non-zero for even or odd times:
    \[
     \mathcal{E}_t(i) = \begin{cases}
        \frac{1}{2t}\, \text{ for even time steps}, \\
        \frac{1}{2t-1}\,  \text{ for odd time steps}.
    \end{cases}
    \]
In the case above we know that the Harmonic series $H_n=\sum_{k=1}^n \frac{1}{k}$ diverges, and since the even and odd numbers perfectly partition the integers $\frac{1}{2}H_n \to \infty$ as well. The event $A_t(i,i)=\mathcal{E}_t(i)=0$ may occur \emph{infinitely} many times yet consensus still follows, provided the minimum-persistence (A3) holds. For instance, every even $t=2\mathbb{N}$ we could have both self-belief and learning zero: \[
A_t(i,i)=\mathcal{E}_t(i)=0.
\]  The learning rate being positive on the odd integer times ensures divergence.
\item 	\emph{Example (finite zeros, assumptions satisfied).}
Fix $T\in\N$. Take $A_t=I, \, \forall t \geq T$ (row-stochastic) and set
\[
\mathcal{E}_t(i)=0\quad\text{for }t\le T,\qquad 
\mathcal{E}_t(i)=\frac{c}{(t+1)\ln(t+2)}\quad\text{for }t>T,
\]
with $0<c\le \tfrac12$. Then $A_t(i,i)\equiv1\ge\alpha=1$, so the event $A_t(i,i)=\mathcal{E}_t(i)=0$ is confined to the finite set $\{0,\dots,T\}$. Also,
$\sum_t \mathcal{E}_{\min,t}=\infty$ by the integral test for $\sum 1/(t\ln t)$, so consensus holds.

\begin{corollary}[Persistence from the tail]\label{cor:persistence-tail}
Let $\mathcal{E}_t(i)=0$ for $t\le T$ and $\mathcal{E}_t(i)=\dfrac{c}{(t+1)\ln(t+2)}$ for $t>T$, with $0<c\le \tfrac12$. Then
\[
\sum_{t=0}^{\infty}\mathcal{E}_{\min,t}
=\sum_{t=T+1}^{\infty} \frac{c}{(t+1)\ln(t+2)}
=c\sum_{t=T+1}^{\infty}\frac{1}{(t+1)\ln(t+2)}
=\infty.
\]
\end{corollary}

\end{enumerate}
\end{remark}

\subsection{Mixed operator norms}
\label{sec:timevarmixed}

The previous result (Theorem~\ref{thm:main_vanish}) establishes that convergence is robust to the temporal decay of the learning signal, provided the cumulative learning impulse diverges. However, this analysis relies primarily on the magnitude of the diagonal entries of $\mathcal{E}_t$. It treats the network structure $A_t$ somewhat passively, requiring only row-stochasticity to maintain the error bounds.

In many dynamic scenarios, the challenge is not just that the learning signal fades, but that the interaction network $A_t$ fluctuates in ways that can transiently amplify errors. As noted in Section~\ref{sec:timvar}, even if each instantaneous update matrix $B_t$ is strictly contractive in terms of spectral radius, their non-commutative product may lead to divergence.

To address this structural instability, we require a more granular analysis of how the update operator $B_t$ reshapes the error distribution. We move beyond simple induced norms (like $\ell_\infty \to \ell_\infty$) and introduce a framework based on \emph{mixed operator norms}. By decomposing the error propagation into two complementary mechanisms: peak-dissipation ($\ell_1 \to \ell_\infty$) and mass-preservation ($\ell_\infty \to \ell_1$). We derive checkable conditions for convergence that exploit the network's mixing ability, providing a stability guarantee that reinforces the minimal persistence results derived above.

Let $M$ be a real square matrix of size $n > 0$ and let $1 \leq p, q \leq \infty$.
We define the following norm, which we generally refer to as the \emph{mixed operator norm}:
\[
\lVert M \rVert_{p \to q} \coloneqq \sup_{\mathbf{v} \in \mathbb{R}^n \backslash \{0\}} \frac{\lVert M \mathbf{v} \rVert_q}{\lVert \mathbf{v} \rVert_p}.
\]
Alternatively, we also call the norm as the matrix $p \to q$ norm.
If $p=q$ then the $p \to q$ norm simply becomes the induced matrix $p$-norm.
By \cite[Theorem 2.1]{Lewis23}, the $p \to q$ norm has explicit formulae for distinct $p, q \in \{1,2,\infty\}$.
Some of them are:
\begin{align}
    \norm{M}_{1 \to \infty} &= \max \left\{ \lvert M(i,j) \rvert : i,j \in \{1,2,\dots,n\} \right\}, \label{eq:1toinf} \\
    \norm{M}_{\infty \to 1} &= \max \left\{ \norm{M \mathbf{v}}_1 : \mathbf{v} \in \{ \pm 1 \}^n \right\}. \label{eq:infto1}
\end{align}
A recent treatment on the problem of estimating the $p \to q$ norm for more general values of $p$ and $q$ can be found in \cite{GMU23}.
The following lemma is a useful property used frequently throughout.
\begin{lemma} \label{lem:MNmixednorms}
Let $M$, $N$ be real square matrices.
Then, for any $1\le p,q,r \le \infty$, we have
\[
\norm{MN}_{q\to p}\;\le\;\norm{M}_{r\to p}\,\norm{N}_{q\to r}
\]
\end{lemma}
\begin{proof}
For any $\mathbf{v} \neq 0$,
\[
\norm{MN \mathbf{v}}_p
\;\le\;\norm{M}_{r\to p}\,\norm{N \mathbf{v}}_r
\;\le\;\norm{M}_{r\to p}\,\norm{N}_{q\to r}\,\norm{\mathbf{v}}_q.
\]
Taking the supremum over $\mathbf{v} \neq 0$ yields the claim.
\end{proof}

Let $t \geq 0$ be an integer and  $V_t \in \mathbb{R}^n$ be the vector such that 
\begin{equation}\label{e:V}
V_t(i) = \lvert A_t(i,i)-\mathcal{E}_t(i) \rvert\text{ for all }i \in \{1,2,\dots,n\}.
\end{equation}
This vector $V_t$ will be frequently used throughout the remainder of this paper.

\begin{lemma} \label{lem:infto1<=}
If $t \geq 0$ is an integer, then
\[
\lVert B_t \rVert_{\infty \to 1} \leq n-\tr A_t + \mathbf{1}^\top V_t.
\]
\end{lemma}

\begin{proof}
Let $i \in \{1,2,\dots,n\}$ and let $\mathbf{v} \in \{\pm 1\}^n$.
By the triangle inequality, we have
\[
\lvert \row_i(B_t) \, \mathbf{v} \rvert \leq \sum_{j=1}^n \lvert B_t(i,j) \rvert = 1-A_t(i,i)+\lvert A_t(i,i)-\mathcal{E}_t(i) \rvert = 1-A_t(i,i) + V_t(i).
\]
It follows that
\[
\lVert B_t \mathbf{v} \rVert_1 = \sum_{i=1}^n \lvert \row_i(B_t) \, \mathbf{v} \rvert \leq n-\tr A_t+\sum_{i=1}^n V_t(i) = n-\tr A_t+\mathbf{1}^\top V_t.
\]
Hence, by \cite[Theorem 2.1]{Lewis23}, we obtain
\[
\lVert B_t \rVert_{\infty \to 1} = \max\{ \Vert B_t \mathbf{v} \rVert_1 : \mathbf{v} \in \{\pm 1\}^n \} \leq n-\tr A_t+\mathbf{1}^\top V_t. \tag*{\qedhere}
\]
\end{proof}

Based on \Cref{lem:infto1<=} above, we will derive conditions such that $\displaystyle \lim_{t \to \infty} Y_t = 0$ where $Y_t \coloneqq X_t-\bar\sigma$.
This mixed-norm approach is better suited to dynamics when the $\infty$-norm fails to produce a contraction.  \Cref{ex:n2} elucidates this approach where two steps forward produce a contraction. However, the actual conditions of the \Cref{thm:normcond} need to be elaborated.

The inequality $\displaystyle \lvert B_t(i,j) \rvert \le  \frac{1}{n-\tau'}$ in \eqref{eq:normcond} below is an upper bound for all entries of $\lvert B_t \rvert$.
Both inequalities counteract each other and there is a balancing tension between them.
To find desirable $\tau$ and $\tau'$, we note that the easier one inequality is to satisfy, the harder the other becomes. The expression $\tr A_t - \mathbf{1}^\top V_t$ can be thought of as the learning mass. If there is no learning in the system, the mass becomes zero. Unlike in the fixed case, in the time-varying case we cannot directly work with the anchoring concept. 
Therefore, to formulate the theorem, it is necessary to introduce precise definitions for the two expressions. Though now if the learning mass is positive, the system has a learner but we cannot ascertain the topology from this condition, hence the smallness of $B_t$ is needed.

\begin{remark}
Mathematically, the necessity of the entry-wise bound $\displaystyle |B_t(i,j)| \le \frac{1}{n-\tau'}$ arises from the need to ensure that the learning influence is ``diffused" correctly through the network. Averaging and its social aspect is important. While the learning mass $\tr A_t - \mathbf{1}^\top V_t$ guarantees that ``energy" is leaving the system, the entry-wise bound ensures that no individual interaction is strong enough to isolate a sub-group from this learning influence. In this sense, the $\infty \to 1$ approach allows us to prove convergence by aggregating the global ``contraction" of the system even when a local contraction (in the standard $\infty$-norm) cannot be guaranteed at every time step.
\end{remark}

\begin{theorem} \label{thm:normcond}
Let $n$ be a positive integer and let $\tau$, $\tau'$ be real numbers such that $0\le \tau' < \tau\le n$.
For each integer $t \geq 0$, suppose that
\begin{equation} \label{eq:normcond}
    \tr A_t - \mathbf{1}^\top V_t \ge  \tau \;\;\, \text{and} \;\;\, \lvert B_t(i,j) \rvert \le  \frac{1}{n-\tau'}
\end{equation}
for all $i, j \in \{1,2,\dots,n\}$.
Then $\displaystyle \lim_{t \to \infty} Y_t = 0$.
\end{theorem}

\begin{proof}
By \cite[Theorem 2.1]{Lewis23} and \eqref{eq:normcond}, we have that
\begin{equation} \label{eq:1toinf<}
    \lVert B_t \rVert_{1 \to \infty} = \max \left\{ \lvert B_t(i,j) \rvert : 1 \leq i,j \leq n \right\} \le \frac{1}{n-\tau'}
\end{equation}
for all integers $t \geq 0$.
By \eqref{eq:normcond} and \Cref{lem:infto1<=}, we obtain
\begin{equation} \label{eq:infto1<=}
    \lVert B_t \rVert_{\infty \to 1} \leq n-\tr A_t+\mathbf{1}^\top V_t \le n-\tau
\end{equation}
for all integers $t \geq 0$.

Let $t \ge 0$ be an integer.
Since $Y_{t+2} = B_{t+1} Y_{t+1} = B_{t+1} B_t Y_t$, by \Cref{lem:MNmixednorms}, we obtain
\begin{equation} \label{eq:normineq}
    \lVert Y_{t+2} \rVert_\infty \leq \lVert B_{t+1} B_t \rVert_{\infty} \cdot \lVert Y_t \rVert_{\infty} \leq \lVert B_{t+1} \rVert_{1 \to \infty} \cdot \lVert B_t \rVert_{\infty \to 1} \cdot \lVert Y_t \rVert_\infty.
\end{equation}
Define $\beta_t \coloneqq \lVert B_{t+1} \rVert_{1 \to \infty} \cdot \lVert B_t \rVert_{\infty \to 1}$.
By \eqref{eq:1toinf<} and \eqref{eq:infto1<=}, we have 
\[
\beta_t \le  \frac{n-\tau}{n-\tau'} < 1.
\]
From \eqref{eq:normineq}, we deduce that if $t > 1$ is an even integer, then
\[
\lVert Y_t \rVert_\infty \leq (\beta_0 \beta_2 \cdots \beta_{t-2}) \cdot \lVert Y_0 \rVert_\infty \le \left( \frac{n-\tau}{n-\tau'} \right)^{t/2} \cdot \lVert Y_0 \rVert_\infty.
\]
Similarly, if $t > 1$ is an odd integer then
\[
\lVert Y_t \rVert_\infty \leq (\beta_1 \beta_3 \cdots \beta_{t-2}) \cdot \lVert Y_1 \rVert_\infty \le \left( \frac{n-\tau}{n-\tau'} \right)^{(t-1)/2} \cdot \lVert Y_1 \rVert_\infty.
\]
In both cases, $\norm{Y_t}_\infty \to 0$ as $t \to \infty$ and therefore, $\displaystyle \lim_{t \to \infty} Y_t = 0$.
\end{proof}

The following theorem generalizes \Cref{thm:normcond} where the quantities $\tau$ and $\tau'$ now depend on the time parameter $t$.

\begin{theorem} \label{thm:normcondgeneral}
Let $\tau_t, \tau'_t$ be real numbers such that $0 \le \tau_t' < \tau_t \le n$.
For each integer $t \geq 0$, suppose that
\begin{equation} \label{eq:normcondgeneral}
    \tr A_t - \mathbf{1}^\top V_t \geq \tau_t \;\;\, \text{and} \;\;\, \lvert B_t(i,j) \rvert \le  \frac{1}{n-\tau'_t}
\end{equation}
for all $i, j \in \{1,2,\dots,n\}$.
Suppose also that
\begin{equation}\label{e:cond_conv}
\prod_{s=0}^{\lfloor t/2 \rfloor-1} \frac{n-\tau_{2s+f(t)}}{n-\tau'_{2s+f(t)+1}} \xrightarrow[t\to\infty]{} 0,
\end{equation}
where $f(t)=0$ if $t$ is even, and $f(t)=1$ otherwise.
Then $\displaystyle \lim_{t \to \infty} Y_t = 0$.
In particular, \eqref{e:cond_conv} is satisfied if there exist $\gamma_0, \gamma_1 \in \mathbb{R}$ such that for all $t \ge 0$,
\begin{equation}\label{e:cond_conv_2}
0 \le \tau'_t\le \gamma_0 < \gamma_1 \le \tau_t \le n.
\end{equation}
\end{theorem}

\begin{proof}
The proof is similar to the proof of \Cref{thm:normcond}.
From \eqref{eq:normcondgeneral}, we have that
\begin{align*}
    \norm{B_t}_{1 \to \infty} &\le \frac{1}{n-\tau'_t} \;,\\
    \norm{B_t}_{\infty \to 1} &\le n-\tr A_t+\mathbf{1}^\top V_t \le n-\tau_t
\end{align*}
for all integers $t \ge 0$.

Let $t \ge 0$ be an integer.
Since $Y_{t+2} = B_{t+1} Y_{t+1} = B_{t+1} B_t Y_t$, by \Cref{lem:MNmixednorms}, we obtain
\begin{equation} \label{eq:normineqgeneral}
    \lVert Y_{t+2} \rVert_\infty \leq \lVert B_{t+1} B_t \rVert_{\infty} \cdot \lVert Y_t \rVert_{\infty} \leq \lVert B_{t+1} \rVert_{1 \to \infty} \cdot \lVert B_t \rVert_{\infty \to 1} \cdot \lVert Y_t \rVert_\infty.
\end{equation}
If $\beta_t \coloneqq \lVert B_{t+1} \rVert_{1 \to \infty} \cdot \lVert B_t \rVert_{\infty \to 1}$ then
\[
\beta_t \le \frac{n-\tau_t}{n-\tau'_{t+1}}.
\]
Using \eqref{eq:normineqgeneral}, we have
\[
\norm{Y_t}_\infty \le \norm{Y_{f(t)}}_\infty \prod_{s=0}^{\lfloor t/2 \rfloor-1} \beta_{2s+f(t)} \le \norm{Y_{f(t)}}_\infty \prod_{s=0}^{\lfloor t/2 \rfloor-1} \frac{n-\tau_{2s+f(t)}}{n-\tau'_{2s+f(t)+1}}.
\]
From \eqref{e:cond_conv}, we conclude that $\displaystyle \lim_{t \to \infty} Y_t = 0$.
Next, suppose that there exist $\gamma_0, \gamma_1 \in \mathbb{R}$ such that for all $t \ge 0$, \eqref{e:cond_conv_2} is satisfied.
Then
\[
\beta_t \le \frac{n-\tau_t}{n-\tau'_{t+1}} \le \frac{n-\gamma_1}{n-\gamma_0} < 1
\]
for all $t$.
Since
\[
\left( \frac{n-\gamma_1}{n-\gamma_0} \right)^{\lfloor t/2 \rfloor} \xrightarrow[t \to \infty]{} 0,
\]
then \eqref{e:cond_conv} holds.
\end{proof}

We illustrate \Cref{thm:normcondgeneral} with \Cref{ex:n2} below.

\begin{example} \label{ex:n2}
Let $t \geq 0$ be an integer.
We let $\displaystyle a_t = \frac{t+3}{3(t+2)}$ and let
\[
A_t = \begin{pmatrix}
    a_t & 1-a_t \\
    1/2 & 1/2
\end{pmatrix},\; \mathcal{E}_t = \begin{pmatrix}
    0 & 0 \\
    0 & 1/2
\end{pmatrix}, \; \text{and} \; B_t = A_t -\mathcal{E}_t = \begin{pmatrix}
    a_t & 1-a_t \\
    1/2 & 0
\end{pmatrix}.
\]
We will see that there exist $\tau'_t$ and $\tau_t$ such that $0 \le \tau'_t < \tau_t \le 2$ and \eqref{eq:normcondgeneral} holds.
First, note that $V(t)=(a_t, 0)^\top$ so $\tr A_t - \mathbf{1}^\top V_t = 1/2$.
If such $\tau_t$ exists then $\tau_t \leq 1/2$.
On the other hand,
\[
\max_{i, j \in \{1,2\}} \lvert B_t(i,j) \rvert = 1-a_t = \frac{2t+3}{3(t+2)}
\]
so if such $\tau'_t$ exists then we must have
\[
\frac{2t+3}{3(t+2)} \le \frac{1}{2-\tau'_t} \implies \tau'_t \ge \frac{t}{2t+3}.
\]
Hence, we must have that
\[
0 \le \frac{t}{2t+3} \le \tau'_t < \tau_t \le \frac{1}{2},
\]
which shows that the desired $\tau'_t$ and $\tau_t$ exist for all $t$.
However, there are no $\gamma_0$, $\gamma_1$ such that \eqref{e:cond_conv_2} is satisfied, since $t/(2t+3)$ can get arbitrarily close to $1/2$.
Instead, we will show that \eqref{e:cond_conv} can hold.
Given an integer $t$, let $f(t)=0$ if $t$ is even, and $f(t)=1$ otherwise.
Fix $\tau'_t = t/(2t+3)$ and $\tau_t = 1/2$, for all $t$.
If $s$ is an integer such that $0 \le s \le \lfloor t/2 \rfloor-1$, then
\[
\frac{2-\tau_{2s+f(t)}}{2-\tau'_{2s+f(t)+1}} = \frac{4s+2f(t)+5}{4s+2f(t)+6} = 1-\frac{1}{4s+2f(t)+6} < 1.
\]
Using $\ln(1-x)\le -x$ for $x \in [0,1]$, we obtain
\[
\ln\prod_{s=0}^{\lfloor t/2 \rfloor-1} \frac{2-\tau_{2s+f(t)}}{2-\tau'_{2s+f(t)+1}} \le -\sum_{s=0}^{\lfloor t/2 \rfloor-1} \frac{1}{4s+2f(t)+6} \xrightarrow[t\to\infty]{} -\infty.
\]
It follows that
\[
\prod_{s=0}^{\lfloor t/2 \rfloor-1} \frac{2-\tau_{2s+f(t)}}{2-\tau'_{2s+f(t)+1}} \xrightarrow[t\to\infty]{} 0
\]
and \eqref{e:cond_conv} is satisfied.
By \Cref{thm:normcondgeneral}, we conclude that $\displaystyle \lim_{t \to \infty} Y_t = 0$.
This example showcases that even if $\lVert B_t \rVert_{\infty} =1$ for all $t$, we can still achieve zero-convergence for $Y_t$.
\end{example}

\begin{remark}[One learner model]
We consider the case when there is only a single learner agent.
If we are to apply \Cref{thm:normcond}, there are some necessary constraints imposed by the assumptions of \Cref{thm:normcond}, which will be examined below.

Let $n>1$ be an integer and let $\tau$, $\tau'$ be real numbers such that $0\le \tau' < \tau\le n$.
Fix $t$ and suppose that only one agent \(i_t\) is a learner, i.e., \(\mathcal{E}_t(i_t)>0\) and \(\mathcal{E}_t(j)=0\) for all \(j\neq i_t\).
Now, suppose that $\tr A_t - \mathbf{1}^\top V_t \ge \tau > 0$ and
\[
\lVert B_t \rVert_{1 \to \infty} = \max_{1 \leq i,j \leq n} \lvert B_t(i,j) \rvert \le \frac{1}{n-\tau'}.
\]
The one learner model will allow us to study the combination of these two conditions.
The quantity
\[
\tr A_t-\mathbf{1}^\top V_t
=\sum_{i=1}^n\!\Big(A_t(i,i)-\big|A_t(i,i)-\mathcal{E}_t(i)\big|\Big)
\]
measures the total margins to anchoring boundaries at time \(t\).
In particular, this quantity in one learner model equals
\[
\tr A_t-\mathbf{1}^\top V_t
= A_t(i_t,i_t)-\big|A_t(i_t,i_t)-\mathcal{E}_t(i_t)\big|
=
\begin{cases}
\mathcal{E}_t(i_t) & \text{if} \;\, \mathcal{E}_t(i_t)\le A_t(i_t,i_t),\\[2mm]
2A_t(i_t,i_t)-\mathcal{E}_t(i_t) & \text{if} \;\, \mathcal{E}_t(i_t) > A_t(i_t,i_t).
\end{cases}
\]
First, we must have $A_t(i_t, i_t) > 0$.
Otherwise, if $A_t(i_t, i_t) = 0$ then $\tr A_t - \mathbf{1}^\top V_t = -\mathcal{E}_t(i_t) < 0$, a contradiction.
Next, it is necessarily true that
\begin{equation} \label{e:0<E<2A}
    0 < \mathcal{E}_t(i_t) < 2A_t(i_t, i_t).
\end{equation}
Otherwise, if $\mathcal{E}_t(i_t) \ge 2A_t(i_t, i_t) > A_t(i_t, i_t)$ then $\tr A_t - \mathbf{1}^\top V_t = 2A_t(i_t, i_t)-\mathcal{E}_t(i_t) \le 0$, a contradiction.

It is obvious that $B_t = A_t-\mathcal{E}_t$ has the same entries as $A_t$ except that $\lvert B_t(i_t, i_t) \rvert < A_t(i_t, i_t)$.
Let $\hat{A}_t$ be the matrix with the same entries as $A_t$ except that $\hat{A}_t(i_t,i_t)=0$.
In this scenario, we have the equality
\[
\|B_t\|_{1\to\infty}
= \max\Big\{\|\hat{A}_t\|_{1\to\infty},\,|B_t(i_t,i_t)|\Big\} =
\max\Big\{\|\hat{A}_t\|_{1\to\infty},\,|A_t(i_t,i_t)-\mathcal{E}_t(i_t)|\Big\}.
\]
Consequently, from $\displaystyle \|B_t\|_{1\to\infty}\le \frac{1}{n-\tau'}<\frac{1}{n-\tau}$, we have
\begin{equation}\label{e:com1_suff}
\max\Big\{\|\hat{A}_t\|_{1\to\infty},\,|A_t(i_t,i_t)-\mathcal E_t(i_t)|\Big\}
\le 
\frac{1}{n-\tau'}<\frac{1}{n-\tau}.
\end{equation}

Notice that because $A_t$ is a row-stochastic matrix, for each row other than $i_t$, its maximum entry is at least $1/n$.
Thus, \begin{equation}\label{e:smallness}
1/n \le \| \hat{A}_t \|_{1\to\infty}
\end{equation}
with equality if and only if: $A_t(i,j)=1/n$ for all $1\le i, j\le n$, $i\ne i_t$ and $A_t(i_t,j)\le 1/n$ for all $j\ne i_t$. 

Since $A_t(i_t, i_t) \le 1$ and $\tr A_t - \mathbf{1}^\top V_t \ge \tau > 0$, we have
\[
\tau \le A_t(i_t,i_t)-\big|A_t(i_t,i_t)-\mathcal{E}_t(i_t)\big| \le 1,
\]
which also means that $\tau \le 1 <n$.
Combined with \eqref{e:com1_suff}, we have
\begin{equation} \label{e:extend}
    \max\Big\{\|\hat{A}_t\|_{1\to\infty},\,|A_t(i_t,i_t)-\mathcal E_t(i_t)|\Big\} < \frac{1}{n-\tau} \le \frac{1}{n-A_t(i_t,i_t)+\big|A_t(i_t,i_t)-\mathcal{E}_t(i_t)\big|} \le \frac{1}{n-1}.
\end{equation}
In particular, $\|\hat{A}_t\|_{1\to\infty} < 1/(n-1)$, which enforces on $\hat{A}_t$ some ``smallness" condition, i.e., all entries of $\hat{A}_t$ must be smaller than $1/(n-1)$.

Using \eqref{e:extend}, we obtain
\begin{equation} \label{e:com2}
|B_t(i_t,i_t)| = \big|A_t(i_t,i_t)-\mathcal{E}_t(i_t)\big|< \min\left\{ r\!\big(A_t(i_t,i_t)\big),\frac{1}{\|\hat{A}_t\|_{1\to\infty}}-n+ A_t(i_t,i_t) \right\},
\end{equation}
where
\[
r(a) \coloneqq \frac{\sqrt{(n-a)^2+4}-(n-a)}{2}.
\]
To see why this is the case, notice that from \eqref{e:extend}, both $\|\hat{A}_t\|_{1\to\infty}$ and $|A_t(i_t,i_t)-\mathcal E_t(i_t)|$ must be less than
\[
\frac{1}{n-A_t(i_t,i_t)+|A_t(i_t,i_t)-\mathcal{E}_{t}|},
\]
which leads to \eqref{e:com2}.
Notice also that \eqref{e:com2} characterizes the admissible range of $\big|A_t(i_t,i_t)-\mathcal{E}_t(i_t)\big|$.
Moreover, as we have mentioned, $\|\hat{A}_t\|_{1\to\infty}\ge 1/n$, and yields in particular that 
\[
0<\mathcal{E}_t(i_t)<2A_t(i_t,i_t),
\]
which is consistent with \eqref{e:0<E<2A}.

Using \eqref{e:com2} and \eqref{e:smallness}, we obtain
\[
\frac{1}{n} \le \|\hat{A}_t\|_{1\to\infty} < \frac{1}{n-A_t(i_t,i_t)} \le \frac{1}{n-1}.
\]
This is a condition on the smallness of the matrix $\hat{A}_t$ that is not automatically guaranteed and means that in general there are restrictions on the matrix $A_t$ in the case of a single learner. We can interpret this also as a diffusion of learning.  In particular, the fact that the entries are not too large forces interaction.
For example, the entries of the non-learner's row must all be positive.
\end{remark}

\section{Models with noise}\label{sec:noise}
We consider two settings. First, the general case of additive noise is studied under the infinity-norm contraction. Second, this result is then combined with the previous section \ref{sec:timevarmixed} to obtain a new robustness result.
\subsection{General additive noise} 
The Theorem below is a generalization  of the original one in \cite{popescu2023averaging}, where the noise was in the feedback or control term. Here it acts an external noise term $\gamma_t$.

The paper \cite{popescu2023averaging} studies noisy feedback of the form
\[
X_t = A_t X_{t-1} + \mathcal{E}_t(\bar\sigma + \gamma_t - X_{t-1}).
\]
Expanding the correction term yields
\[
X_t = A_tX_{t-1} + \mathcal{E}_t(\bar\sigma - X_{t-1}) + \mathcal{E}_t\gamma_t,
\]
and hence, for $Y_t \coloneqq X_t-\bar\sigma$,
\[
Y_t = (A_t-\mathcal{E}_t)Y_{t-1} + \mathcal{E}_t\gamma_t.
\]
Thus PV's model induces an additive disturbance $r_t \coloneqq \mathcal{E}_t\gamma_t$ in the error recursion.
In the present work we allow a general exogenous perturbation $r_t$ added directly to the state,
\[
X_{t+1}=A_tX_t+\mathcal{E}_t(\bar\sigma-X_t)+r_t,
\]
which is strictly more general unless $\mathcal{E}_t$ is invertible (or $r_t$ lies in the range of $\mathcal{E}_t$
coordinatewise). In the Theorem~\ref{t:35} below, we refer to the general $\gamma_t$ as additive external noise.

\noindent\textbf{Setup and role of the assumptions.}
Throughout we assume $\bar\sigma=\sigma\mathbf 1$. Since each $A_t$ is row-stochastic, $A_t\bar\sigma=\bar\sigma$,
and the centered error $Y_t\coloneqq X_t-\bar\sigma$ satisfies
\[
Y_t=(A_t-\mathcal E_t)Y_{t-1}+\gamma_t=:B_tY_{t-1}+\gamma_t.
\]
Let $\rho_t\ge \|B_t\|_{\infty\to\infty}$. The small-gain condition \eqref{eq:42} provides uniform control of the
geometric tail $\Lambda_t=\rho_t+\rho_t\rho_{t-1}+\cdots+\rho_t\cdots\rho_1$, which implies exponential bounds on
products $\prod_{i=s+1}^t\rho_i$ (Lemma~\ref{lem1}). This guarantees that the initial error is forgotten and, if
$\gamma_t\to0$, consensus $X_t\to\bar\sigma$ follows from the variation-of-constants bound. In part~(iv), the same
contraction mechanism is applied on $\mathcal P_1(\mathbb R^n)$ under $W_1$-type metrics to obtain convergence in law \cite{Santambrogio2015OptimalTransport,Villani09}.

\begin{theorem}\label{t:35}
Assume
\[
X_t \;=\; A_t X_{t-1} + \mathcal{E}_t(\bar\sigma - X_{t-1})+\gamma_t,
\]
with $A_t$ deterministic and row-stochastic and $\mathcal{E}_t$ also deterministic diagonal nonnegative matrix. Let $B_t\coloneqq A_t-\mathcal{E}_t$ and select $\rho_t$ such that
\[
\rho_t
\ge \;\|B_t\|_{\infty\to\infty}
=\| A_t-\mathcal{E}_t \|_{\infty}
=\max_{i=1,\dots,n}\Big(\, |A_t(i,i)-\mathcal{E}_{t}(i)| + 1 - A_t(i,i)\,\Big).
\]
Assume
\begin{equation}\label{eq:42}
\sup_{t\ge 1}\{\rho_t + \rho_t\rho_{t-1} + \cdots + \rho_t\rho_{t-1}\cdots\rho_1\}\;<\;\infty.
\end{equation}
Then:
\begin{enumerate}
\item[(i)] If $\gamma_t\to 0$ a.s., then $X_t\to\bar\sigma$ a.s.
\item[(ii)] If $\gamma_t\to 0$ in probability, then $X_t\to\bar\sigma$ in probability.
\item[(iii)] If $\gamma_t\to 0$ in $L^p$, then $X_t\to\bar\sigma$ in $L^p$.
\item[(iv)] Suppose $Y_t=B_tY_{t-1}+\gamma_t$, with $\gamma_t$
integrable and independent of $\mathcal{F}_{t-1}$ with $\mathcal{F}_t=\sigma(\gamma_s :s\le t)$. In addition to \eqref{eq:42}, assume
\begin{equation}\label{eq:44}
B_t \xrightarrow[t\to\infty]{}B
\end{equation}
for some matrix $B$.

If in addition we have that 
\begin{equation}\label{eq:UI}
W_1(\gamma_t,\gamma)\xrightarrow[t\to\infty]{}0,
\end{equation}
then there exists an $n$-dimensional random vector $Y$ such that $Y_t\Rightarrow Y$.
\item[(v)] Moreover, if the limit $\gamma$ in  \eqref{eq:UI} is integrable and non-degenerate a.s., then without \eqref{eq:44} the conclusion in (iv) may fail.
\end{enumerate}
\end{theorem}

We should point out that part (iv) here is a more general version of \cite[Theorem 4.1 part (iv)]{popescu2023averaging} and with a different proof included here.  

\begin{proof}
Before we jump into the technical part, let's work out some notations and preliminary observations. 

Since $A_t$ is row–stochastic, $A_t\bar\sigma=\bar\sigma$. Set $Y_t\coloneqq X_t-\bar\sigma$ and $B_t\coloneqq A_t-\mathcal{E}_t$. Then
\[
Y_t \;=\; B_t Y_{t-1} + \gamma_t
\]
and, by the induced operator norm,
\begin{equation}\label{eq:one-step}
\|Y_t\|_\infty \;\le\; \rho_t\,\|Y_{t-1}\|_\infty + \|\gamma_t\|_\infty,\qquad \rho_t\coloneqq\|B_t\|_{\infty}.
\end{equation}
Iterating (with the convention $\prod_{i=t+1}^t\rho_i=1$) gives the variation–of–constants bound
\begin{equation}\label{eq:VoC}
\|Y_t\|_\infty \;\le\; \Big(\prod_{s=1}^t \rho_s\Big)\,\|Y_0\|_\infty
\;+\; \sum_{s=0}^{t}\Big(\prod_{i=t-s+1}^{t}\rho_i\Big)\,\|\gamma_{t-s}\|_\infty.
\end{equation}

We use the following results from \cite{popescu2023averaging}.

\begin{lemma}[\cite{popescu2023averaging}]\label{lem1}
Let $\Lambda_t\coloneqq\rho_t+\rho_t\rho_{t-1}+\cdots+\rho_t\rho_{t-1}\cdots\rho_1$ and assume
$\sup_{t\ge 1}\Lambda_t\le \Lambda<\infty$.
Set
\[
c\coloneqq\log\Big(1+\frac{1}{\Lambda}\Big),\qquad K\coloneqq 1+\Lambda.
\]
Then, for all $0\le s<t$,
\begin{equation}\label{eq:A3}
\prod_{i=s+1}^{t}\rho_i \;\le\; K e^{-c(t-s)}.
\end{equation}
Moreover, for all $s\ge1$ and $t\ge s$,
\begin{equation}\label{eq:A5}
\prod_{i=t-s+1}^{t}\rho_i\Big(1+\rho_{t-s}+\rho_{t-s}\rho_{t-s-1}+\cdots+\rho_{t-s}\cdots\rho_{1}\Big)
\;\le\;K^2e^{-cs}.
\end{equation}
Conversely, if \eqref{eq:A3} holds for some $K,c>0$, then
\[
\sup_{t\ge 1}\Lambda_t\le \frac{K}{e^c-1}.
\]
\end{lemma}
\begin{proof}
Define $\Lambda_0\coloneqq 0$. Since $\Lambda_t=\rho_t(1+\Lambda_{t-1})$, we have
$\rho_t=\Lambda_t/(1+\Lambda_{t-1})$.
Therefore, for $0\le s<t$,
\[
\prod_{i=s+1}^{t}\rho_i
=\prod_{i=s+1}^{t}\frac{\Lambda_i}{1+\Lambda_{i-1}}
=\frac{\Lambda_t}{1+\Lambda_s}\prod_{i=s+1}^{t-1}\frac{\Lambda_i}{1+\Lambda_i}.
\]
Using $\Lambda_t\le \Lambda$ and $\frac{\Lambda_i}{1+\Lambda_i}\le \frac{\Lambda}{1+\Lambda}=e^{-c}$,
\[
\prod_{i=s+1}^{t}\rho_i
\le \Lambda\Big(\frac{\Lambda}{1+\Lambda}\Big)^{t-s-1}
=(1+\Lambda)\Big(\frac{\Lambda}{1+\Lambda}\Big)^{t-s}
=Ke^{-c(t-s)},
\]
which proves \eqref{eq:A3}. For \eqref{eq:A5}, note that the parenthesis equals
$1+\Lambda_{t-s}\le 1+\Lambda=K$, hence multiplying \eqref{eq:A3} (with $s$ replaced by $t-s$)
gives \eqref{eq:A5}.

Conversely, assume \eqref{eq:A3} holds for some $K,c>0$. For each $t\ge1$, by definition,
\[
\Lambda_t
=\rho_t+\rho_t\rho_{t-1}+\cdots+\rho_t\rho_{t-1}\cdots\rho_1
=\sum_{k=1}^{t}\prod_{i=t-k+1}^{t}\rho_i .
\]
Fix $k\in\{1,\dots,t\}$ and set $s\coloneqq t-k$ so that $t-s=k$ and
$\prod_{i=t-k+1}^{t}\rho_i=\prod_{i=s+1}^{t}\rho_i$. Applying \eqref{eq:A3} yields
\[
\prod_{i=t-k+1}^{t}\rho_i \le K e^{-c(t-s)} = K e^{-ck}.
\]
Therefore,
\[
\Lambda_t \le \sum_{k=1}^{t} K e^{-ck} \le \sum_{k\ge1} K e^{-ck}
= K\frac{e^{-c}}{1-e^{-c}}=\frac{K}{e^c-1}.
\]
In particular, $\sup_t\Lambda_t\le K/(e^c-1)$.

\end{proof}

\begin{enumerate}

\item[(i)-(iii)]
All these items are identical in proof with the one from \cite[Theorem 4.1]{popescu2023averaging}.  

\item[(iv)] This has a different proof from \cite[Theorem 4.1 part (iv)]{popescu2023averaging}. Write $X_t=B_tX_{t-1}+\gamma_t$ with $B_t\coloneqq A_t-\mathcal{E}_t$. Work on $\mathcal P_1(\mathbb R^n)$ with
\[
D_\infty(\mu,\nu)\coloneqq\inf_{\alpha}\int\|x-y\|_\infty\,\alpha(dx,dy)=\inf\mathbb E\|Z-Z'\|_\infty,
\]
the $W_1$ metric induced by $\|\cdot\|_\infty$ (equivalent to the usual $W_1$). 

For the rest of the paper, we denote $\mathcal{L}(X)$ to be the law (distribution) of the random variable $X$.  

For any $\mu,\nu$ and $Z\sim\mu$, $Z'\sim\nu$, coupling $Z,Z'$ optimally and using the \emph{same} $\gamma_t$ gives the one–step contraction
\begin{equation}\label{eq:W1-contract}
D_\infty\big(\Law(B_tZ+\gamma_t),\,\Law(B_tZ'+\gamma_t)\big)\ \le\ \rho_t\,D_\infty(\mu,\nu).
\end{equation}

This mirrors \cite[A.11 in the Appendix]{popescu2023averaging} though the proof here is slightly different and more streamlined.   Notice the key condition here is that $\gamma_t$ is independent of $\mathcal{F}_{t-1}$ which is used to justify \eqref{eq:W1-contract}, particularly that the distribution of $(Z,\gamma_t)$ is constructed and compared to the distribution of $(Z',\gamma_t)$.  If the noise $\gamma_t$ were not independent of $\mathcal{F}_t$, we would have had problems defining things properly.

Next, we argue that by \eqref{eq:44}, $B_t\to B$ in $\|\cdot\|_\infty$. We claim $\|B\|_{\infty}<1$.
Indeed, if $\|B\|_{
\infty
}\ge1$, for any $1>\delta>0$, we have that $\|B_t\|_{\infty}\ge 1-\delta$ for large enough $t$, say $t\ge t_\delta$.  In particular, for $t-s$ large with $t\ge s\ge t_\delta$,
\[
(1-\delta)^{t-s}\le \prod_{i=s+1}^t \rho_i. 
\]
On the other hand, from \Cref{lem1} we get that for some constants $K>0$ and $c>0$ and all large $t-s$ that 
\[
(1-\delta)^{t-s}\le \prod_{i=s+1}^t \rho_i\le K e^{-c(t-s)}.
\]
Consequently we get that $1-\delta\le e^{-c}$ for any $\delta>0$ which is a contradiction.  Hence $\|B\|<1$.

Define the \emph{limit} kernel
\[
\mathcal T_\infty(\mu)\coloneqq\Law(BZ+\gamma),
\]
where $\mu$ is the distribution of $Z$.

Then $\mathcal T_\infty$ is a strict $\!D_\infty$–contraction with constant $\|B\|<1$, so it has a unique fixed point $\nu_\ast$.

Assume \eqref{eq:UI} which also implies that $\gamma_t\Rightarrow\gamma$. Then, also $D_\infty(\gamma_t,\gamma)\to0$.
Let $\mathcal T_t(\mu)\coloneqq\Law(B_tZ+\gamma_t)$. For any $\mu=\mathcal{L}(Z)$,
\begin{equation}\label{eq:kernel-perturb}
\begin{aligned}
D_\infty(\mathcal T_t\mu,\mathcal T_\infty\mu)
&\le D_\infty\!\big(\Law(B_tZ+\gamma_t),\Law(BZ+\gamma_t)\big)
 + D_\infty\!\big(\Law(BZ+\gamma_t),\Law(BZ+\gamma)\big)\\
&\le \|B_t-B\|_\infty\,\mathbb E\|Z\|_\infty \;+\; D_\infty(\gamma_t,\gamma).
\end{aligned}
\end{equation}
From \eqref{eq:VoC}, \Cref{lem1} and \eqref{eq:UI}, we have $\sup_t\mathbb E\|Y_t\|_\infty<\infty$.
Let $\mu_t\coloneqq\Law(Y_t)$ and $\nu_t\coloneqq\mathcal T_\infty^{\,t}(\mu_0)$.  Here we use $\mu_0=\mathcal{L}(Y_0)$ for the distribution of the initial condition. Using \eqref{eq:W1-contract} for $\mathcal T_\infty$ and \eqref{eq:kernel-perturb} with $\mu=\mu_{t-1}$,
\[
D_\infty(\mu_t,\nu_t)\ \le\ \|B\|\,D_\infty(\mu_{t-1},\nu_{t-1})
\;+\;\underbrace{\|B_t-B\|_\infty\,\mathbb E\|Y_{t-1}\|_\infty + D_\infty(\gamma_t,\gamma)}_{=:\alpha_t}.
\]
By \eqref{eq:44} the first term in $\alpha_t$ tends to $0$ and the second tends to $0$ by \eqref{eq:UI}. 

The next step is to invoke the following simple result.

\begin{lemma}\label{lem3}
Let $(d_t)_{t\ge1}$ and $(\alpha_t)_{t\ge1}$ be nonnegative  and satisfy $d_t\le \rho_t d_{t-1}+\alpha_t$ with $(\rho_t)$ obeying \eqref{eq:42} and $\alpha_t \xrightarrow[t\to\infty]{}  0$. Then $d_t\xrightarrow[t\to\infty]{} 0$.
\end{lemma}

\begin{proof}
Unwinding the recursion and using \eqref{eq:A3} we can justify for $1\le s\le t$
\[
\begin{split}
d_t\ &\le\ \Big(\prod_{i=1}^t\rho_i\Big)d_0\ +\ \sum_{j=0}^{t-1}\Big(\prod_{i=t-j+1}^{t}\rho_i\Big)\alpha_{t-j}\le K e^{-c(t-1)}d_0+ K\sum_{j=0}^{t-1}e^{-cj}\alpha_{t-j}\\ 
&\le K e^{-c(t-1)}d_0+ K\sum_{j=0}^{s}e^{-cj}\alpha_{t-j}+K\sum_{j=s+1}^{t}e^{-cj}\alpha_{t-j}\\
&\le K e^{-c(t-1)}d_0+ K\frac{\max_{t-s\le j\le t} {\alpha_j}}{1-e^{-c}}+\frac{Ke^{-c(s+1)}}{1-e^{-c}}\sup_{j\ge0}\alpha_j
\end{split}
\]
Now letting $t$ tend to infinity and using that $\alpha$ converges to $0$, we get that 
\[
0\le \limsup_{t} d_t\le \frac{Ke^{-c(s+1)}}{1-e^{-c}}\sup_{j\ge0}\alpha_j. 
\]
Now letting $s$ to infinity we conclude that $d_t$ converges to $0$. 

We thus get that 
\[
D_\infty(\mu_t,\nu_t) \xrightarrow[t\to\infty]{} 0,
\]
which amounts to   $\mu_t\Rightarrow\nu_\ast$ and thus the proof of this item of the Theorem.
\end{proof}

\item[(v)] Now we discuss the necessity of \eqref{eq:44}.

We construct a one-dimensional counterexample (external noise) with $\rho_t\le b_2<1$ so \eqref{eq:42} holds, but \eqref{eq:44} fails and $\Law(X_t)$ does not converge.

Take $n=1$, $A_t\equiv1$, and pick $0<b_1<b_2<1$. Let $B_t\in\{b_1,b_2\}$ on alternating blocks whose lengths $L_m\uparrow\infty$; then $B_t$ does not converge (violates \eqref{eq:44}) while $\rho_t=|B_t|\le b_2<1$ (so \eqref{eq:42} holds). Let $(\gamma_t)$ be i.i.d., integrable, non-degenerate. For fixed $b\in(0,1)$, the affine kernel $\mathcal T_b(\mu)=\Law(bZ+\gamma)$ is a strict $W_1$–contraction with unique fixed point $\pi_b$; after $L$ steps at coefficient $b$,
\[
W_1\big(\Law(X_{\text{end}}),\pi_b\big)\ \le\ b^{\,L}W_1\big(\Law(X_{\text{start}}),\pi_b\big).
\]
Choose $L_m$ with $b_j^{L_m}\le 2^{-m}$ on each $b_j$–block. Along block endpoints, the laws approach alternately $\pi_{b_1}$ and $\pi_{b_2}$.  Next we prove that the distributions of $\pi_b$ and $\pi_{b'}$ are different.  

\begin{lemma}
Let $(\gamma_t)_{t\ge0}$ be iid and non-degenerate.  Set $X_b=\sum_{k\ge0}b^k\gamma_k$ for $0<b<1$.  Then the distribution of  $X_b$ is $\pi_b$ and for $0<b\ne b'<1$, $\pi_b\ne \pi_{b'}$.  
\end{lemma}

\begin{proof}
It is elementary to check that $\pi_b$ is the distribution of $X_b=\sum_{k\ge0}b^k\gamma_k$.  

If the distribution of $\gamma_t$ is not degenerate and  $b\ne b'$, then $\pi_b$ is not the same as $\pi_{b'}$.  Indeed, if they have the same distribution, then the characteristic functions are the same.  

Let $\varphi(\xi)$ be the characteristic function  of $\gamma_0$, and let $\Psi_b(\xi)$ be the characteristic function of $\pi_b$. The random variable $X_b$ satisfies the distributional fixed-point equation $X_b \stackrel{d}{=} \gamma_0 + b X_b$. In terms of characteristic functions, this implies:
\[
\Psi_b(\xi) = \varphi(\xi) \Psi_b(b\xi).
\]
Assume, for the sake of contradiction, that $\pi_b = \pi_{b'}$ with $0 < b < b' < 1$. This implies $\Psi_b(\xi) = \Psi_{b'}(\xi)$ for all $t$. Substituting this into the fixed-point equation for both $b$ and $b'$:
\[
\varphi(\xi) \Psi_b(b\xi) = \Psi_b(\xi) = \Psi_{b'}(\xi) = \varphi(\xi) \Psi_b(b'\xi).
\]
Since $\varphi(\xi)$ is continuous and $\varphi(0)=1$, there exists a neighborhood $(-\delta, \delta)$ where $\varphi(t) \neq 0$. For $\xi$ in this neighborhood, we can divide by $\varphi(\xi)$:
\[
\Psi_b(bt) = \Psi_b(b't).
\]
Let $\xi = b't$ and $\lambda = b/b'$. Since $0 < b < b'$, we have $0 < \lambda < 1$. The equation becomes $\Psi_b(\lambda \xi) = \Psi_b(\xi)$. Iterating this relation $n$ times yields:
\begin{equation}\label{e:chid}
\Psi_b(\xi) = \Psi_b(\lambda^n \xi).
\end{equation}
Taking the limit as $n \to \infty$, and using the continuity of characteristic functions:
\[
\Psi_b(\xi) = \lim_{n \to \infty} \Psi_b(\lambda^n \xi) = \Psi_b(0) = 1.
\]
Combining this with \Cref{e:chid}, we can argue that $\Psi_b(\xi) = 1$ for any $|\xi|<\delta$, thus $\pi_b$ is degenerate at $0$ (i.e., $X_b = 0$ almost surely). By the definition $X_b = \gamma_0 + bX_b$, this implies $\gamma_0 = 0$ almost surely. This contradicts the assumption that $\gamma_t$ is non-degenerate. Therefore, $\pi_b \neq \pi_{b'}$. \qedhere
\end{proof}
Hence $\Law(X_t)$ has two distinct subsequential limits and does not converge. This proves the necessity claim.\qedhere
\end{enumerate}
\end{proof}

\Cref{t:35} provides the contraction machinery needed for the robustness result below, where we combine it with the mixed-norm bounds of \cref{sec:timevarmixed}.

\subsection{A stochastic variant}
In this setting, the averaging–learning recursion is influenced by an \emph{exogenous disturbance} (or forcing) term:
\[
X_{t+1}=A_tX_t+\mathcal{E}_t(\bar\sigma-X_t)+r_t.
\]
Here, $r_t$ models external shocks that are not produced by the averaging or learning dynamics themselves (for instance, shared environmental influences,  or common noise). Our objective is to demonstrate that the mixed-norm two-step contraction developed earlier leads to a robust stability result for the centered process

\[
Y_t:=X_t-\bar\sigma,
\]
which evolves according to $Y_{t+1}=B_tY_t+r_t$, where $B_t:=A_t-\mathcal{E}_t$.

The contraction bounds control how past shocks accumulate. If the disturbance vanishes (almost surely or in
probability), then the process tracks the noiseless limit and $X_t\to\bar\sigma$ in the corresponding
mode; this yields parts~(a)--(b) of Theorem~\ref{thm:norm:noisy-prob}. If the disturbance persists, pointwise
convergence is typically not the right notion. Instead, under asymptotic time-homogeneity $B_t\to B$ and an
asymptotically stationary noise assumption in Wasserstein distance, the law of $X_t$ converges to a limiting
distribution; this is made precise in part~(c). In this sense, Theorem~\ref{thm:norm:noisy-prob} is a robust
counterpart of Theorem~\ref{thm:normcondgeneral}: vanishing shocks recover the deterministic limit, whereas persistent shocks lead to convergence in
distribution to a limiting law.
\begin{theorem}[Noisy convergence under probabilistic modes]\label{thm:norm:noisy-prob}
Take the recursion
\[
X_{t+1}\;=\;A_tX_t+\mathcal{E}_t(\bar\sigma-X_t)+r_t,\qquad t=0,1,2,\dots
\]
and let $Y_t\coloneqq X_t-\bar\sigma$ and $B_t\coloneqq A_t-\mathcal{E}_t$.  Also let
\[
V_t(i)\coloneqq\big|A_t(i,i)-\mathcal{E}_t(i)\big|,\qquad i=1,\dots,n.
\]
Let $\tau_t,\tau'_t$ be real numbers such that $0\le \tau'_t<\tau_t\le n$, and assume that for every $t\ge0$,
\begin{equation}\label{eq:normcondgeneral-noisy}
\tr A_t-\mathbf{1}^\top V_t\ge \tau_t
\qquad\text{and}\qquad
|B_t(i,j)|\le \frac{1}{n-\tau'_t}\;\;\;\forall\,i,j\in\{1,\dots,n\}.
\end{equation}
Define the two-step ``gains''
\[
\rho^{(0)}_k\coloneqq\frac{n-\tau_{2k}}{n-\tau'_{2k+1}},
\qquad
\rho^{(1)}_k\coloneqq\frac{n-\tau_{2k+1}}{n-\tau'_{2k+2}},
\qquad k=0,1,2,\dots
\]
and assume the (Theorem~\ref{t:35}-type) summability bounds
\begin{equation}\label{eq:rho-sum-bounded}
\sup_{k\ge1}\Big\{\rho^{(\ell)}_k+\rho^{(\ell)}_k\rho^{(\ell)}_{k-1}+\cdots+\rho^{(\ell)}_k\rho^{(\ell)}_{k-1}\cdots\rho^{(\ell)}_1\Big\}<\infty,
\qquad \ell\in\{0,1\}.
\end{equation}

Then:

\begin{enumerate}
\item[\textup{(a)}] 
If $r_t\to0$ almost surely, then $X_t\to\bar\sigma$ almost surely.

\item[\textup{(b)}] 
If $r_t\to0$ in probability, then $X_t\to\bar\sigma$ in probability.

\item[\textup{(c)}] 
Assume $\{r_t\}_{t\ge0}$ are independent and integrable and 
\begin{equation}\label{eq:slow-variation-noisy}
B_t\xrightarrow[t\to\infty]{}B.
\end{equation}
If moreover there exists an integrable $r$ such that
\begin{equation}\label{eq:W1-noise-conv}
W_1(r_t,r)\xrightarrow[t\to\infty]{}0,
\end{equation}
then there exists an $n$-dimensional random vector $X_\infty$ such that $X_t\Rightarrow X_\infty$ as $t\to\infty$.
Moreover, without \eqref{eq:slow-variation-noisy} the conclusion can fail even with integrable i.i.d.\ noise unless $r_t$ is a.s.\ constant.
\end{enumerate}
\end{theorem}

\begin{lemma}\label{lem:block-rec}
Let $Y_{t+1}=B_tY_t+r_t$.
Define the even and odd blocks
\[
Z^{(0)}_k\coloneqq Y_{2k},\qquad Z^{(1)}_k\coloneqq Y_{2k+1}.
\]
Then
\begin{align*}
Z^{(0)}_{k+1} &= C^{(0)}_k Z^{(0)}_k + \gamma^{(0)}_k,
\qquad
C^{(0)}_k\coloneqq B_{2k+1}B_{2k},\quad \gamma^{(0)}_k\coloneqq B_{2k+1}r_{2k}+r_{2k+1},\\
Z^{(1)}_{k+1} &= C^{(1)}_k Z^{(1)}_k + \gamma^{(1)}_k,
\qquad
C^{(1)}_k\coloneqq B_{2k+2}B_{2k+1},\quad \gamma^{(1)}_k\coloneqq B_{2k+2}r_{2k+1}+r_{2k+2}.
\end{align*}
Moreover, under \eqref{eq:normcondgeneral-noisy},
\[
\|C^{(0)}_k\|_\infty\le \rho^{(0)}_k,
\qquad
\|C^{(1)}_k\|_\infty\le \rho^{(1)}_k.
\]
\end{lemma}

\begin{proof}
The block recursions follow by expanding two steps:
\[
Y_{2k+2}=B_{2k+1}Y_{2k+1}+r_{2k+1}
=B_{2k+1}(B_{2k}Y_{2k}+r_{2k})+r_{2k+1}.
\]
The odd block is analogous.

For the gain bounds, use mixed-norm submultiplicativity (\Cref{lem:MNmixednorms}) with $p=q=\infty$ and $r=1$:
\[
\|C^{(0)}_k\|_\infty=\|B_{2k+1}B_{2k}\|_{\infty\to\infty}
\le \|B_{2k+1}\|_{1\to\infty}\,\|B_{2k}\|_{\infty\to1}.
\]
Thus, we obtain
\[
\|C^{(0)}_k\|_\infty\le \frac{1}{n-\tau'_{2k+1}}\,(n-\tau_{2k})=\rho^{(0)}_k.
\]
The odd case is identical.
\end{proof}


\begin{proof}[Proof of Theorem~\ref{thm:norm:noisy-prob}]
Since $A_t$ is row-stochastic, $A_t\bar\sigma=\bar\sigma$, hence the error satisfies
\[
Y_{t+1}=B_tY_t+r_t.
\]

\emph{Step 1 (reduce to block recursions and verify the Theorem~\ref{t:35} hypothesis).}
By \Cref{lem:block-rec}, both block processes satisfy
\[
Z^{(\ell)}_{k+1}=C^{(\ell)}_k Z^{(\ell)}_k+\gamma^{(\ell)}_k,\qquad \ell\in\{0,1\},
\]
with $\|C^{(\ell)}_k\|_\infty\le \rho^{(\ell)}_k$.  The boundedness condition
\eqref{eq:rho-sum-bounded} is exactly the hypothesis \eqref{eq:42} of Theorem~\ref{t:35}
applied to each block recursion (with $\rho_k$ replaced by $\rho^{(\ell)}_k$).

\emph{Step 2 (parts (a) and (b): vanishing noise).}
Assume first $r_t\to0$ a.s.
$\sup_t\|B_t\|_\infty<\infty$ (indeed $\|B_t\|_\infty\le n$), hence
\[
\|\gamma^{(0)}_k\|_\infty\le \|B_{2k+1}\|_\infty\,\|r_{2k}\|_\infty+\|r_{2k+1}\|_\infty \xrightarrow[k\to\infty]{}0
\quad\text{a.s.}
\]
(and similarly for $\gamma^{(1)}_k$).  Applying Theorem~\ref{t:35}(i) to each block recursion gives
$Z^{(0)}_k\to0$ a.s.\ and $Z^{(1)}_k\to0$ a.s., i.e.\ $Y_{2k}\to0$ and $Y_{2k+1}\to0$ a.s., so $Y_t\to0$ a.s.
Thus $X_t\to\bar\sigma$ a.s.

If $r_t\to0$ in probability, the same bound implies $\gamma^{(\ell)}_k\to0$ in probability for $\ell\in\{0,1\}$,
and Theorem~\ref{t:35}(ii) yields $Z^{(\ell)}_k\to0$ in probability, hence $Y_t\to0$ in probability and $X_t\to\bar\sigma$
in probability.

\emph{Step 3 (part (c): convergence in distribution under persistent noise).}
Assume now (c).  From $B_t\to B$, each block coefficient sequence $C^{(\ell)}_k$ satisfies, hence $C^{(\ell)}_k\to C^{(\ell)}$ for some limit matrix (and necessarily $C^{(0)}=C^{(1)}=B^2$).

Moreover, $\gamma^{(\ell)}_k$ is integrable since $r_t$ is integrable and $\sup_t\|B_t\|_\infty<\infty$.
Finally, from $B_{2k+1}\to B$ and $W_1(r_t,r)\to0$, standard stability of $W_1$ under Lipschitz maps
implies $W_1(\gamma^{(\ell)}_k,\gamma)\to0$ for the limit noise
\[
\gamma \;\stackrel{d}{=}\; B\widetilde r+\widetilde r',
\]
where $\widetilde r,\widetilde r'$ are independent with law $r$ (matching the two fresh noises inside each block).
(One can prove this by coupling $r_{2k}$ with $\widetilde r$ and $r_{2k+1}$ with $\widetilde r'$ and using that
$x\mapsto Bx$ is $\|B\|_\infty$-Lipschitz in $\|\cdot\|_\infty$.)

Therefore, Theorem~\ref{t:35}(iv) applies to each block recursion and yields
\[
Y_{2k}=Z^{(0)}_k \Rightarrow Y_\infty
\quad\text{and}\quad
Y_{2k+1}=Z^{(1)}_k \Rightarrow Y_\infty
\]
for some random vector $Y_\infty$ (indeed both parities converge to the \emph{same} limit because in the limit the
two-step kernel is $\mu\mapsto \Law(B^2Z+B\widetilde r+\widetilde r')$, which is a strict $W_1$-contraction under
\eqref{eq:rho-sum-bounded}, hence has a unique fixed point; both parities must converge to that unique fixed point).
Thus $Y_t\Rightarrow Y_\infty$ and $X_t=\bar\sigma+Y_t\Rightarrow \bar\sigma+Y_\infty=:X_\infty$.

The final ``moreover'' statement (failure without slow variation) follows by applying
Theorem~\ref{t:35}(v) to an appropriate one-dimensional block recursion (alternating stable coefficients on longer and longer blocks),
exactly as in the counterexample of Theorem~\ref{t:35}(v).
\end{proof}

\section{Conclusion}
In this paper, we have shown that the Averaging-Learning dynamics allow for quite nuanced behaviour. Agents can be social but not learn through the learning-rate matrix $\mathcal{E}$. This, however, imposes conditions on the averaging matrix $A$. The first half of the paper developed the concept of anchoring, using graph theory. The condensed anchoring property is not only new but useful to analyze networks where not all agents have access to perfect learning. In this sense, even when some agents are defective, society as a whole can converge to the truth.

When the network varies in time, spectral methods fail. We introduced a mixed-operator-norm framework that extracts two-step contraction from composing $\norm{\cdot}_{1 \to \infty}$ and $\norm{\cdot}_{\infty \to 1}$ bounds. No single step needs be contractive. This machinery is new to the consensus literature and applies broadly to products of time-varying sub-stochastic matrices. This mechanism is robust to vanishing noise and preserves convergence; persistent noise drives the process to a limiting law.

Several questions remain open. Problem \ref{prob:genconstantbound} asks whether a learning rate of $3$ for any agent forces instability. A positive answer would give a sharp threshold for overlearning. The gap between necessity and sufficiency for the zero-convergence of $A-\mathcal{E}$ when $A - \mathcal{E}$ has negative diagonal entries also deserves further study.

\appendix

\section{Appendix: The Open Problem} \label{appendix}

We would like to discuss related questions to the open Problem~\ref{prob:genconstantbound}. 

To make this section more readable we include the problem again here.  

\begin{problem}\label{}
Is it true that $\rho(A-\mathcal{E}) \geq 1$ for $\mathcal{E}$ an $n\times n$ diagonal matrix such that for some $i\in\{1,2\dots, n\}$, $\mathcal{E}(i)\ge 3$ and $A$ is any $n\times n$ row-stochastic matrix?      
\end{problem}

The first round is about some easier cases which can be tackled and then proceed with a more quantitative version of it.

In the following theorem, we have that if $A$ is a symmetric row-stochastic matrix with an agent of learning rate at least $2$, then $A-\mathcal{E}$ is not zero-convergent.

\begin{theorem} \label{thm:symmstoch2}
    Let $A$ be a symmetric row-stochastic matrix of size $n > 0$ and let $\mathcal{E}$ be a diagonal nonnegative matrix of size $n$.
    Suppose that there exists an $i \in \{1,2,\dots,n\}$ such that $\mathcal{E}(i) \geq 2$.
    Then $\rho(A-\mathcal{E}) \geq 1$.
\end{theorem}

\begin{proof}
    If $\lambda$ is an eigenvalue of $A$ then $\lambda \le 1$.
    On the other hand, if $\lambda$ is the smallest eigenvalue of $-\mathcal{E}$ then $\lambda \le -\mathcal{E}(i) \le -2$.
    By the Weyl's theorem~\cite[Theorem~4.3.1]{HJ85}, if $\lambda$ is the smallest eigenvalue of $A-\mathcal{E}$ then $\lambda \le -2+1 =-1$, which implies that $\rho(A-\mathcal{E}) \geq \lvert \lambda \rvert \geq 1$.
\end{proof}

\begin{remark}
    Theorem~\ref{thm:symmstoch2} does not contradict Theorem~\ref{thm:extendcondenseanchor}.
    Suppose that both assumptions of \Cref{thm:symmstoch2} and \Cref{thm:extendcondenseanchor} are true.
    There exists an $i$ such that $\mathcal{E}(i) \ge 2$ and hence, $A(i,i)=1$.
    It follows that $\{i\}$ is a sink SCC where $i$ is a non-anchor, so $A-\mathcal{E}$ cannot be condensely anchored.
\end{remark}

If the matrix $A$ in the assumption of Theorem~\ref{thm:symmstoch2} is not symmetric, then it is possible that $\rho(A-\mathcal{E})<1$.
For example, let
\[
A = \begin{pmatrix}
    1/5 & 0 & 4/5 \\
    1 & 0 & 0 \\
    0 & 1/5 & 4/5 \\
\end{pmatrix}
\; \text{and} \;\, \mathcal{E} = \begin{pmatrix}
    1/2 & 0 & 0 \\
    0 & 0 & 0 \\
    0 & 0 & 2 \\
\end{pmatrix}.
\]
Then $\rho(A-\mathcal{E}) < 0.92$.

\begin{proposition}
    Let $A$ be a row-stochastic matrix of size $n > 0$ and let $\mathcal{E}$ be a diagonal nonnegative matrix of size $n$.
    Suppose that there exists an $i \in \{1,2,\dots,n\}$ such that $\mathcal{E}(i) \geq 2$.
    Suppose also that $0 \leq \mathcal{E}(j) < 2A(j,j)$ for all $j \in \{1,2,\dots,n\}$ where $j \neq i$.
    Then $\rho(A-\mathcal{E}) \geq 1$.
\end{proposition}

\begin{proof}
    Note that
    \[
    (A(i,i)-\mathcal{E}(i))+(1-A(i,i)) = 1-\mathcal{E}(i) \leq -1
    \]
    while for all $j \in \{1,2,\dots,n\}$ where $j \neq i$, we have
    \[
    (A(j,j)-\mathcal{E}(j))-(1-A(j,j)) = 2A(j,j)-\mathcal{E}(j)-1 > -1.
    \]
    This tells us that the Gershgorin disk centered at $A(i,i)-\mathcal{E}(i)$ is disjoint from the rest of the Gershgorin disks centered at $A(j,j)-\mathcal{E}(j)$, where $j \neq i$.
    Hence, there exists precisely one eigenvalue $\lambda$ of $A-\mathcal{E}$ such that $\lambda$ belongs to the Gershgorin disk centered at $A(i,i)-\mathcal{E}(i)$, and therefore, $\rho(A-\mathcal{E}) \ge |\lambda| \geq 1$.
\end{proof}

\subsection{The quantitative version}

\begin{definition}[The extremal function $f_n(R)$]\label{def:fn}
Fix $n\ge 1$ and $R\ge 0$. Define
\[
f_n(R)\;:=\;\inf_{\mathcal{E},A}\ \rho(\mathcal{E}-A),
\]
where $\mathcal{E}=\diag(\mathcal{E}(1),\dots,\mathcal{E}(n))$ is diagonal with $\max_i \mathcal{E}(i)=\mathcal{E}(1)=R$ and
$A\in\R^{n\times n}$ is row substochastic.
\end{definition}

\begin{remark}
If $A$ is row-stochastic and $\mathcal{E}$ is diagonal, then lower bounds on $f_n(R)$
translate into lower bounds for $\rho(A-\mathcal{E})$ when $\max_i\mathcal{E}(i)=R$.
\end{remark}

\begin{problem}[Quantitative form of Problem~\ref{prob:genconstantbound}] \label{conj:dimensionfree}
There exists $R_0>0$ and $c>1$ such that for all $R\ge R_0$,
\[
\inf_{n\ge 1} f_n(R)\ \ge\ c.
\]
In particular, one expects $f_n(3)\ge 1$ for all $n$ (which would resolve
Problem~\ref{prob:genconstantbound} with the threshold $3$).
\end{problem}
\subsection{Dimension dependent estimates for \texorpdfstring{$f_n(R)$}{fn(R)}}\label{sec:general-bound}

We want to give some dimension dependent versions of the quantity $f_n(R)$.
We start by reproducing (and slightly adapting) the classical perturbation bound of
Ostrowski--Elsner to the $\|\cdot\|_\infty$ norm; see \cite{Ostrowski57}, \cite{Elsner1985OptimalSpectralVariation},
and the expositions in \cite{StewartSun1990MatrixPerturbationTheory} and
\cite{Gil2021PerturbationSurvey}.


\subsubsection{An \texorpdfstring{$\|\cdot\|_\infty$}{infinity-norm} Ostrowski--Elsner bound (determinant proof)}

Recall the one-sided \emph{spectral variation} (cf.\ \cite{Elsner1985OptimalSpectralVariation,Gil2021PerturbationSurvey})
\[
s_A(B)\;\coloneqq\;\max_{\mu\in\spec(B)}\min_{\lambda\in\spec(A)}|\lambda-\mu|,
\]
and the Hausdorff distance
$\mathrm{hd}(\spec(A),\spec(B))=\max\{s_A(B),\,s_B(A)\}$
(see, e.g., \cite[Chapter~IV]{StewartSun1990MatrixPerturbationTheory}).

\medskip

\begin{lemma}[Distance-to-spectrum via a determinant]\label{lem:dist-det}
Let $A\in\mathbb C^{n\times n}$ with eigenvalues $\lambda_1,\dots,\lambda_n$
(counted with algebraic multiplicity). Then for every $z\in\mathbb C$,
\[
\min_{1\le i\le n}|z-\lambda_i|^n
\;\le\;\prod_{i=1}^n|z-\lambda_i|
\;=\;|\det(zI-A)|.
\]
\end{lemma}

\begin{proof}
The identity $\prod_{i=1}^n(z-\lambda_i)=\det(zI-A)$ is the characteristic polynomial factorization.
The inequality is immediate since the minimum factor is at most the geometric mean.
This step appears routinely in determinant-based eigenvalue perturbation arguments; see, e.g.,
\cite[Equation~(29)]{Gil2021PerturbationSurvey}.
\end{proof}

\medskip

\begin{lemma}[A determinant-difference bound in $\|\cdot\|_\infty$]\label{lem:det-diff-infty}
For any $X,Y\in\mathbb C^{n\times n}$,
\[
|\det X-\det Y|
\;\le\; n\,\|X-Y\|_\infty\,\max\{\|X\|_\infty,\|Y\|_\infty\}^{\,n-1}.
\]
\end{lemma}

\begin{proof}
Write $X$ and $Y$ row-wise: $X=(x_1^\top;\dots;x_n^\top)$ and $Y=(y_1^\top;\dots;y_n^\top)$.
Using multilinearity of the determinant in the rows,
\[
\det X-\det Y
=\sum_{k=1}^n \det\bigl(y_1^\top;\dots;y_{k-1}^\top;(x_k-y_k)^\top;x_{k+1}^\top;\dots;x_n^\top\bigr).
\]
For each term, apply Hadamard's inequality
\cite[Section~7.8]{HJ85} in the form
\[
|\det Z|\le \prod_{i=1}^n \| \text{row}_i(Z)\|_2
\;\le\;\prod_{i=1}^n \| \text{row}_i(Z)\|_1,
\qquad\text{and note}\qquad \|\text{row}_i(Z)\|_1\le \|Z\|_\infty.
\]
Thus each summand is bounded by
\[
\|x_k-y_k\|_1 \cdot \max\{\|X\|_\infty,\|Y\|_\infty\}^{\,n-1}
\;\le\;\|X-Y\|_\infty \cdot \max\{\|X\|_\infty,\|Y\|_\infty\}^{\,n-1}.
\]
Summing over $k=1,\dots,n$ gives the claim.
\end{proof}

\medskip

\begin{theorem}[Ostrowski--Elsner bound in $\|\cdot\|_\infty$]\label{thm:OE-infty}
Let $A,B\in\mathbb C^{n\times n}$ and let $\|\cdot\|_\infty$ be the induced matrix infinity norm.
Then
\[
s_A(B)\;\le\; n^{1/n}\,\|A-B\|_\infty^{1/n}\,
\Bigl(\|A\|_\infty+\|B\|_\infty\Bigr)^{1-1/n}.
\]
Consequently,
\[
\mathrm{hd}(\spec(A),\spec(B))
\;\le\; n^{1/n}\,\|A-B\|_\infty^{1/n}\,
\Bigl(\|A\|_\infty+\|B\|_\infty\Bigr)^{1-1/n}.
\]
\end{theorem}

\begin{proof}
The argument follows the determinant-based perturbation strategy going back to
Ostrowski~\cite{Ostrowski57} and used in Elsner's bound
\cite{Elsner1985OptimalSpectralVariation}; see also \cite{Gil2021PerturbationSurvey}.

Fix any eigenvalue $\mu\in\spec(B)$.
Applying Lemma~\ref{lem:dist-det} with $z=\mu$,
\[
\min_{\lambda\in\spec(A)}|\lambda-\mu|^n \;\le\; |\det(\mu I-A)|.
\]
Since $\mu\in\spec(B)$, we have $\det(\mu I-B)=0$, hence
\[
|\det(\mu I-A)|
=|\det(\mu I-A)-\det(\mu I-B)|.
\]
Apply Lemma~\ref{lem:det-diff-infty} with $X=\mu I-A$ and $Y=\mu I-B$:
\[
|\det(\mu I-A)|
\le n\,\|A-B\|_\infty\,\max\{\|\mu I-A\|_\infty,\|\mu I-B\|_\infty\}^{\,n-1}.
\]
Now $\|\mu I-A\|_\infty\le |\mu|+\|A\|_\infty$ and $\|\mu I-B\|_\infty\le |\mu|+\|B\|_\infty$.
Also, for any induced norm, $|\mu|\le \rho(B)\le \|B\|_\infty$
(cf.\ \cite[Section~5.6]{HJ85}), hence
\[
\max\{\|\mu I-A\|_\infty,\|\mu I-B\|_\infty\}\le \|A\|_\infty+\|B\|_\infty.
\]
Combining,
\[
\min_{\lambda\in\spec(A)}|\lambda-\mu|^n
\le n\,\|A-B\|_\infty\,(\|A\|_\infty+\|B\|_\infty)^{n-1}.
\]
Taking $n$th roots and then the maximum over $\mu\in\spec(B)$ gives the bound on $s_A(B)$.
The Hausdorff bound follows since $\mathrm{hd}=\max\{s_A(B),s_B(A)\}$.
\end{proof}

\subsubsection{Application to \texorpdfstring{$f_n(R)$}{fn(R)} and a threshold \texorpdfstring{$R_n$}{Rn} for \texorpdfstring{$f_n(R)>1$}{fn(R)>1}}

Recall
\[
f_n(R)\;=\;\inf_{\mathcal{E},A}\rho(A-\mathcal{E}),
\]
where $\mathcal{E}=\mathrm{diag}(\mathcal{E}(1),\dots,\mathcal{E}(n))$ with $\max_i \mathcal{E}(i)=\mathcal{E}(1)=R$ and $A$ is substochastic (row sums $\le 1$).

Set
\[
M\coloneqq A-\mathcal{E},\qquad N\coloneqq -\mathcal{E}.
\]
Then $\rho(\mathcal{E}-A)=\rho(A-\mathcal{E})=\rho(M)$ and $M-N=A$.

Because $A$ is substochastic, $\|A\|_\infty\le 1$.
Moreover $\|N\|_\infty=\|\mathcal{E}\|_\infty=R$ and one checks row-wise that
\[
\|M\|_\infty=\|A-D\|_\infty\le R+1.
\]
Apply Theorem~\ref{thm:OE-infty} to $A=M$ and $B=N$:
\[
s_M(N)\;\le\; n^{1/n}\,\|M-N\|_\infty^{1/n}\,(\|M\|_\infty+\|N\|_\infty)^{1-1/n}
\;\le\; n^{1/n}\,(2R+1)^{1-1/n}.
\]
Since $-R\in\spec(N)$, there exists $\lambda\in\spec(M)$ such that
\[
|\lambda+R|\;\le\; n^{1/n}(2R+1)^{1-1/n}.
\]
Therefore
\begin{equation}\label{e:sp_1}
\rho(M)\;\ge\;|\lambda|\;\ge\; R-n^{1/n}(2R+1)^{1-1/n}.
\end{equation}
Taking the infimum over all admissible $(D,A)$ yields the general lower bound.

\begin{theorem}\label{t:lbf_n}
For every $n\ge 1$ and $R\ge 0$,
\begin{equation}
\label{eq:fn-lower}
f_n(R)\;\ge\;\Bigl(R-n^{1/n}(2R+1)^{1-1/n}\Bigr)^+.
\end{equation}
In particular, set
\[
\bar R_n \;:=\; n2^{\,n-1}+\frac{3n-1}{2} \;=\;\frac{n2^n+3n-1}{2}.
\]
Then for every $R>\bar R_n$ we have $f_n(R)>1$.
Consequently, if $R_n$ denotes the (unique) solution of
$R-n^{1/n}(2R+1)^{1-1/n}=1$, then $R_n\le \bar R_n$.
Moreover, since $\bar R_n\le n(2^{n-1}+2)$ for all $n\ge 1$,
the simpler sufficient condition $R\ge n(2^{n-1}+2)$ also implies $f_n(R)>1$.
\end{theorem}

\begin{proof}
The bound \eqref{eq:fn-lower} is exactly \eqref{e:sp_1}.

For the threshold, define
\[
g(R):=R-n^{1/n}(2R+1)^{1-1/n}.
\]
We want $g(R)>1$. Let $y:=(2R+1)^{1/n}$ so that $R=(y^n-1)/2$ and
\[
g(R)>1
\quad\Longleftrightarrow\quad
\frac{y^n-1}{2}-n^{1/n}y^{n-1}>1
\quad\Longleftrightarrow\quad
y^{n-1}\bigl(y-2n^{1/n}\bigr)>3.
\]

Let $a:=2n^{1/n}$. For any $y\ge a$,
\[
y^n-a^n=(y-a)\sum_{k=0}^{n-1}y^{n-1-k}a^k
\;\le\; n y^{n-1}(y-a),
\]
hence
\[
y^{n-1}(y-a)\;\ge\;\frac{y^n-a^n}{n}.
\]
Now take $R=\bar R_n$, so $y^n=2\bar R_n+1=n2^n+3n$ while $a^n=n2^n$, giving
\[
\frac{y^n-a^n}{n}=\frac{3n}{n}=3
\quad\Longrightarrow\quad
y^{n-1}(y-a)\ge 3
\quad\Longrightarrow\quad
g(\bar R_n)\ge 1.
\]
If $R>\bar R_n$, then $y^n-a^n>3n$ and the same chain yields
$y^{n-1}(y-a)>3$, i.e. $g(R)>1$, hence by \eqref{eq:fn-lower} we get $f_n(R)>1$.

Finally, $\bar R_n\le n(2^{n-1}+2)$ is immediate from
$\frac{3n-1}{2}\le 2n$.
\end{proof}

\subsection{\texorpdfstring{$f_n(R)$}{fn(R)} for \texorpdfstring{$n=1,2$}{n=1,2}}

We now consolidate the exact piecewise formulas and upper bounds for $f_n(R)$ when $n \in \{1, 2\}$ into a single unified theorem with streamlined proofs.

\begin{theorem}[Formulas and bounds for $f_n(R)$ for $n=1,2$]\label{thm:fn-small-n}
For $R \ge 0$, 
\[
f_1(R) = f_2(R) = (R-1)^+.
\]
\end{theorem}

\begin{proof}
First, observe that for any dimension $n$, if $0 \le R \le 1$, choosing the matrix $D = R I$ and $A = D$ ensures $A \ge 0$, row sums equal to $R \le 1$, and $M = D - A = 0$. Hence, $\rho(M) = 0$ and $f_n(R) = 0$. We now consider the regime $R > 1$.

\paragraph{Cases $n=1, 2$:}
For $n=1$, the matrix is a scalar $M = [R - A_{11}]$. With $0 \le A_{11} \le 1$, the minimum modulus is precisely $(R-1)^+$, trivially achieved at $A_{11}=\min(1, R)$.

For $n=2$, the strict upper bound is attained by the decoupled blocks $\mathcal{E} = \diag(R, 0)$ and $A = \diag(\min(1, R), 0)$. 

Now, for the lower bound, notice that the eigenvalues of $M$ are given by 
$$ \lambda_{\pm} = \frac{M_{11} + M_{22} \pm\sqrt{(M_{11} - M_{22})^2 + 4M_{12}M_{21}}}{2}. $$
Because $A_{12}=M_{12}\ge 0$ and $A_{21}=M_{21}\ge0$, the square root is strictly larger than $|M_{11}-M_{22}|$.  This gives us an unbreakable algebraic inequality:
$$ \lambda_+ \ge \frac{M_{11} + M_{22} + |M_{11} - M_{22}|}{2}.$$ 
Therefore:
$$ \lambda_+ \ge \max(M_{11}, M_{22}) \ge M_{11} = R - A_{11}.$$

Since $A_{11}\ge0$, this forces $\lambda_+\ge R-1$.  
\end{proof}

\small
\bibliographystyle{plainurl}
\bibliography{ReferencesJT.bib}

\end{document}